\documentclass[11pt]{amsart}
\pdfoutput=1

\usepackage[utf8]{inputenc} \usepackage[T1]{fontenc}
\usepackage{lmodern}

\usepackage{amsmath, amssymb, amsthm, mathtools, amsbsy, bbm, amsfonts, enumitem, color, url, enumitem, tikz, wrapfig}

\usepackage[all]{xy}
\usepackage[english]{babel}
\usepackage{graphicx, color}
\DeclareGraphicsExtensions{.pdf,.jpg}
\usepackage{xcolor}
\usepackage[pdfpagelabels]{hyperref}
\usepackage[Option]{overpic}

\theoremstyle{plain}
\newtheorem{thm}{Theorem}[section]
\newtheorem*{thm*}{Theorem}
\newtheorem{prop}[thm]{Proposition}
\newtheorem{lemma}[thm]{Lemma}
\newtheorem*{lemma*}{Lemma}

\newtheorem{corollary}[thm]{Corollary}

\newtheorem{notation}[thm]{Notation}

\theoremstyle{definition}
\newtheorem{definition}[thm]{Definition} 
\newtheorem*{definition*}{Definition} 
\newtheorem{example}[thm]{Example}

\theoremstyle{remark}
\newtheorem{remark}[thm]{Remark}

\numberwithin{equation}{thm}

\newcommand{\Z}{\mathbb{Z}}

\newcommand{\define}{\mathrel{\mathop:}=}
\newcommand{\longrightharpoonup}{\relbar\joinrel\rightharpoonup}
\newcommand{\pfold}[1]{\mathbin{\raisebox{-0.1em}{\ensuremath{\overset{#1}{\longrightharpoonup}}}}}
\newcommand{\Shadow}{\mathrm{Sh}} 
\newcommand{\id}{{\bf{1}}} 

\newcommand{\Dir}{{\mathrm{Dir}}}
\renewcommand{\v}{{\mathrm{v}}}

\newcommand{\type}{\tau} 
\newcommand{\dtype}{\hat{\tau}} 
\newcommand{\foot}{\mathrm{ft}} 
\newcommand{\folds}{\mathrm{F}} 
\newcommand{\Ch}{\mathrm{Ch}} 
\newcommand{\aW}{W} 
\newcommand{\sW}{W_0} 
\newcommand{\sS}{S_0} 
\newcommand{\hyp}{\mathcal{H}} 
\newcommand{\shyp}{\mathcal{S}} 
\newcommand{\Cf}{\mathrm{C}_0} 

\usepackage{cleveref}

\begin{document}
	
\title[Shadows in Coxeter groups]{Shadows in Coxeter groups}
\author{Marius Graeber}
\address{Karlsruhe Institute of Technology, Englerstrasse 2, 76137 Karlsruhe, Germany}
\email{marius.graeber@kit.edu}

\author{Petra Schwer}
\address{Otto-von-Guericke University Magdeburg, IAG, Postschließfach 4120, 39016 Magdeburg, Germany}
\email{petra.schwer@ovgu.de}

\date{ \today }
\thanks{}

\begin{abstract}
For a given $w$ in a Coxeter group $W$ the elements $u$ smaller than $w$ in Bruhat order can be seen as the end-alcoves of stammering galleries of type $w$ in the Coxeter complex $\Sigma$. We generalize this notion and consider sets of end-alcoves of galleries that are positively folded with respect to certain orientation $\phi$ of $\Sigma$. We call these sets \emph{shadows}. Positively folded galleries are closely related to the geometric study of affine Deligne-Lusztig varieties, MV polytopes, Hall-Littlewood polynomials and many more agebraic structures. \newline 
In this paper we will introduce various notions of orientations and hence shadows and study some of their algorithmic properties.

\end{abstract}

\maketitle

\section{Introduction}\label{introduction}

It is well known that the Bruhat order on a Coxeter group $(W,S)$ has a geometric interpretation in terms of galleries: the set of all elements $y\leq x$ for a fixed $x\in W$ is the set of all end-alcoves of \emph{folded} (or stammering) galleries of type $x$ in the Coxeter complex $\Sigma=\Sigma(W,S)$. One can show that for given $x, y\in W$ one has  $y\leq x$ in Bruhat order if and only if  there exists a folded  gallery of type $x$ which ends in $y$. 

In the present paper we look at sets of end-alcoves of folded galleries where the foldings are positive with respect to a given orientation $\phi$ of the complex $\Sigma$. Such galleries will be called \emph{$\phi$-positively folded}. 
An orientation on a Coxeter complex essentially decides for every pair of an alcove and a hyperplane containing one of its co-dimension one faces, whether or not the alcove lies on a positive side of the hyperplane. 

The notion of a positively folded gallery goes back to \cite{GaussentLittelmann} (respectively \cite{Littelmann1}). This concept requires a refined notion of what is typically known as a gallery in a Coxeter complex, namely in addition to the sequence of alcoves a gallery contains one needs to remember a specific codimension one face of any two subsequent alcoves. This is equivalent to a choice of a decorated word in $S$, i.e. a word plus the knowledge at which positions the corresponding gallery stammers.

(Positively) folded galleries and paths have appeared in several places some of which we will now highlight. 
Folded paths were used to compute Hall-Littlewood polynomials by C. Schwer in \cite{Schwer}.  Kapovich  and Millson study folded paths in connection with their proof of the saturation conjecture for $SL_n$ in  \cite{KM}. 
Ehrig \cite{Ehrig} studies MV-polytopes by means of Bruhat-Tits buildings and gives a type-independent definition of MV-polytopes by assigning to every LS-gallery in the sense of \cite{GaussentLittelmann} an explicitly constructed MV-polytope.   
There are probably other references we have missed. 

The aim of the present paper is to extract and generalize some of the combinatorics contained in the aforementioned applications and the joint work of the second author with Mili\'{c}evi\'{c} and Thomas on affine Deligne--Lusztig varieties \cite{MST}. We would like to make these folding games accessible on a purely combinatorial level while at the same time providing tools for future applications in other areas of mathematics. There is upcoming work by the second author together with 
Mili\'{c}evi\'{c}, Naqvi and Thomas
\cite{MNST} in which shadows are studied further and will be related to retractions from infinity based at chimneys as well as double coset intersections.
Additionally shadows have direct connections with MV-cycles and polytopes and non-emptiness of affine Deligne-Lusztig varieties. It is for example interesting to see (and no coincidence) that the length additivity condition in Theorem~\ref{thm:partial_shadow} also appears in work of Mili\'{c}evi\'{c} (Beazley), see Theorem 1.4 in \cite{Beazley}.   

\begin{wrapfigure}[17]{r}[0pt]{0.5\textwidth} 
	\begin{center}
		\vspace{-2ex}
		\includegraphics[width=0.4\textwidth, angle=180]{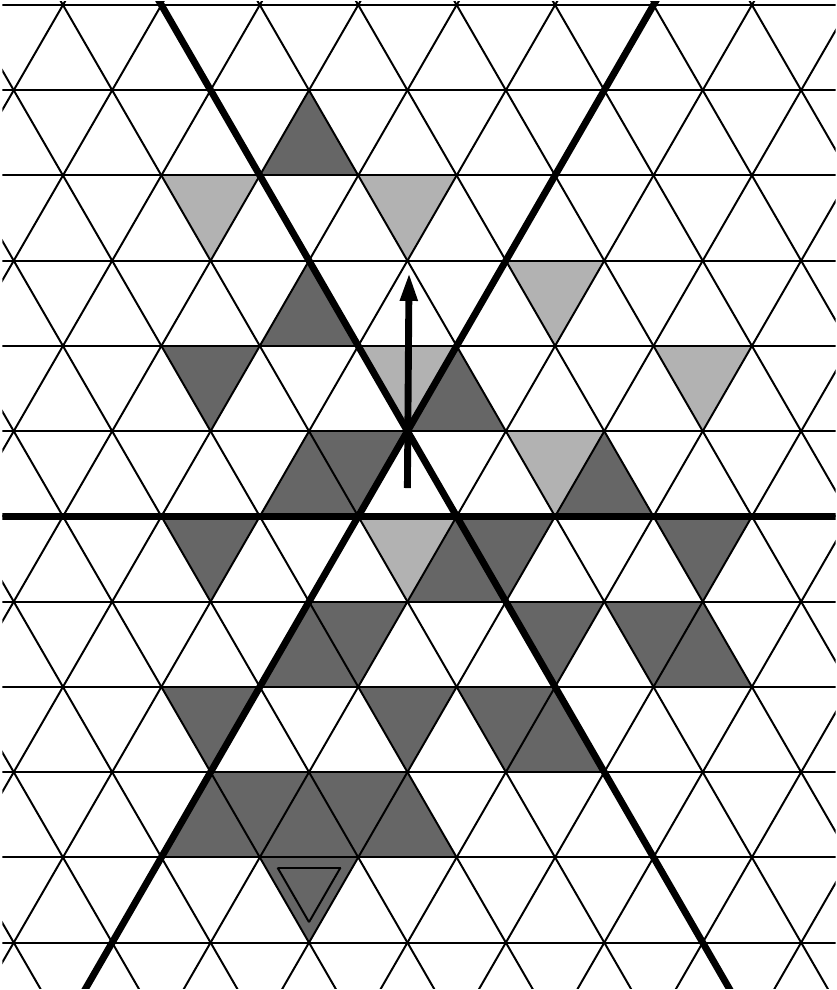}
		\vspace{2ex}
		\caption{A shadow in type $\tilde{A}_2$. }
		\label{fig:shadows-soft-hard-1}
	\end{center}
\end{wrapfigure}

The main concept of the paper is the notion of a \emph{shadow}, which we formally introduce in Definition~\ref{def:shadow}. The shadow of $w$ with respect to some orientation $\phi$ is the set of end-alcoves of all galleries of type $w$ that are $\phi$-positively folded.  

We will study a natural class of orientations, the so called Weyl chamber orientations, which is induced by a choice of a regular direction or, equivalently, by a parallel class of Weyl chambers. Our main results are recursive descriptions of shadows with respect to these Weyl chamber orientations. See Theorems~\ref{thm:regular_shadow} and~\ref{thm:partial_shadow}.  

An example for a shadow with respect to a Weyl chamber orientation is shown in Figure~\ref{fig:shadows-soft-hard-1}. This picture illustrates the full and regular shadows in a type $\tilde A_2$ Coxeter group of the outlined alcove at the top with respect to the orientation determined by the regular vector. Details are explained in Example~\ref{ex:shadows}. 

{\textbf{This article is organized as follows:}}
We use the second section to fix notation for several basic facts on Coxeter groups. Orientations on Coxeter complexes and some of their properties are discussed in Section~\ref{sec:orientations}, where we also define the notion of a regular orientation. 
Folded galleries, ways to manipulate them as well as some statistics on the number of folds are discussed in Section~\ref{sec:folded galleries}. In Section~\ref{sec:shadows} we then define the central notion of the present paper: shadows. 
Section~\ref{sec:regular shadows} finally contains the algorithms and recursive descriptions of regular shadows and their restricted cousins. 

\thanks{The second author would like to thank Anne Thomas and Jacinta Torres for helpful comments and Elizabeth Mili\'{c}evi\'{c} for her thoughtful remarks on an earlier draft of the paper.}

%
%
\section{Coxeter systems and Coxeter complexes}\label{sec:Coxeter groups}

\label{sec:Coxeter}

We assume that the reader is familiar with the standard notions and objects associated to Coxeter groups. For details please refer to one of the many good textbooks on the topic; for example \cite{BjoernerBrenti, Davis} or \cite{Humphreys}.  

Throughout this paper $(W,S)$ will denote a Coxeter system. We will write $u, v, w$ for words in the generators $S$ of $W$ and $[u],[v], [w]$ for the associated elements in $W$. In general elements in $W$ will be denoted by $x,y,z$. 
Any subset $S'\subset S$ defines a \emph{standard parabolic} subgroup $W_{S'}$ of $W$ and each pair $(W_{S'}, S')$ is a Coxeter system in its own right. 

For a given Coxeter system write $\Sigma=\Sigma(W,S)$ for the set of all left-cosets $xW_{S'}$ of standard parabolic subgroups in $W$ which is partially ordered by reverse inclusion and hence  forms an abstract simplicial complex. The vertex set of $\Sigma$ is the set containing all cosets of maximal parabolic subgroups corresponding to subsets $S'=S\setminus\{s\}$. 
The maximal simplices in the \emph{Coxeter complex} $\Sigma$ are called \emph{alcoves} and their codimension one faces \emph{panels}. We will typically denote alcoves by $c,d$ and panels by $p,q$. Note that each panel $p$ corresponds to a coset of a parabolic subgroup of the form $xW_{\{s\}}$ for some $s\in S$. In this case we say $p$ has \emph{type} $s$ and write $\type(p)=s$. 

The group $W$ contains a subset $R\define\bigcup_{x\in W} xSx^{-1}$ of \emph{reflections} each of which fixes a hyperplane (or wall) in $\Sigma$. For a given reflection $r\in R$ we denote the associated hyperplane by $H_r$. 
We say that a hyperplane $H$ \emph{separates} alcoves $c$ and $d$ if the two alcoves are contained in different half-spaces determined by $H$.  

In case that $(\aW, S)$ is a euclidean Coxeter system of type $\widetilde X$ the group $\aW$ splits as a semi-direct product of a spherical Weyl group $\sW$ of type $X$ and a translation group $T$ acting on $\Sigma$. 
The set of \emph{special} vertices in $\Sigma$ are the ones whose stabilizer in $\aW$ is isomorphic to $\sW$.  
In this setting $\Sigma$ does have a geometric realization as a tiled euclidean $n$-space with $n=\# S-1$ if $(W,S)$ is irreducible. The group $T$ is isomorphic to $\Z^n$ and corresponds to the co-root lattice. By slight abuse of notation we denote the geometric realization of $\Sigma$ also by $\Sigma$. 

Fix a special vertex $0$ and call it the \emph{origin} of $\Sigma$.   For each special vertex $v$ in the orbit of $0$ under $T$ consider the set $\mathcal{H}_v$ of hyperplanes through $v$. The closures of the connected components of $\Sigma \setminus \cup_{H\in\mathcal{H}_v} H$ are called \emph{Weyl chambers} in $\Sigma$. \newline
The set of equivalence classes of parallel rays in $\Sigma$ form the boundary sphere $\partial\Sigma$. This sphere inherits a natural tiling from  $\Sigma$ by taking as the hyperplanes in $\partial \Sigma$ the parallel classes of hyperplanes in $\Sigma$. The maximal simplices in $\partial\Sigma$ then are precisely the parallel classes of Weyl chambers in $\Sigma$. We sometimes refer to the maximal simplices in the boundary as \emph{chambers} in order to distinguish them from alcoves in $\Sigma$. As a simplicial complex $\partial\Sigma$ is isomorphic to the Coxeter complex of $(\sW, \sS)$ where $\sS$ is a subset of $S$ generating a copy of $\sW$.  

We will choose the identifications of elements in an affine Coxeter group $\aW$ with the alcoves in $\Sigma$ and the identification of element in the associated $\sW$ with chambers in $\partial \Sigma$ in a compatible way. The identity in $\sW$ labels a chamber at infinity which has a unique representative $\Cf$  with basepoint $0$ in $\Sigma$, the \emph{fundamental Weyl chamber}. The unique alcove in $\Cf$ containing $0$ is labeled with $\id$. Then the $\aW$ action on $\Sigma$ yields identifications of elements $x\in \aW$ with alcoves in $\Sigma$. The walls of $\Cf$ correspond to the generators in $S$ that also generate $\sW$. The equivalence class of a Weyl chamber $x.\Cf$ with cone point $0$ has label $x$ in $\sW$. That is the image of some $x\in\aW$ under the natural projection $p:\aW \to \sW$ can be interpreted both as the local spherical direction of an element $x=ty$ with $t\in T$ and $y\in\sW$ and as the direction at infinity towards which $y$ points when seen as an alcove with basepoint $t.0$.

%
%
\section{Orientations on Coxeter complexes}\label{sec:orientations}
In this section we will introduce orientations of Coxeter complexes and provide some natural examples. We start with the definition and some basic properties in the first subsection below. 

\subsection{General notions}

If not otherwise stated $(W,S)$ is any Coxeter system and $\Sigma$ is its associated Coxeter complex. 

\begin{definition}[Orientations of $\Sigma$]
	An \emph{orientation} $\phi$ of $\Sigma$ is a map which assigns to a pair of a panel $p$ and an alcove $c$ containing $p$ a value in $\{+1, -1\}$.  We say that $c$ is on the \emph{$\phi$-positive side} (respectively the \emph{$\phi$-negative side}) of $p$ if $\phi(p,c)=+1$ (respectively -1).
\end{definition}

\begin{example}[Trivial orientations]
	One way to produce an orientation is to take the map $\phi$ to be a constant map which is either $\equiv +1$ or $\equiv -1$. We will refer to these orientations as the \emph{trivial positive/negative orientation}. 
\end{example}

Sometimes we will want to exclude orientations which locally behave like trivial ones and therefore introduce the following two notions. 

\begin{definition}[Locally nonnegative/nontrivial orientations]
	An orientation $\phi$ of $\Sigma$ is called 
	\begin{enumerate}[label=(\roman*)]
		\item \emph{locally nonnegative} if every panel $p$ has at least one $\phi$-positive side.
		\item \emph{locally nontrivial} if every panel $p$ has exactly one $\phi$-positive side.
	\end{enumerate}
\end{definition}


The Coxeter group $W$ naturally acts on the set of all orientations of the associated Coxeter complex. 

\begin{definition}[$W$-action on orientations]
	Let $(W, S)$ be a Coxeter system with Coxeter complex $\Sigma$. Then the natural left action of $W$ on the alcoves and panels of $\Sigma$ induces a natural left action of $W$ on the orientations of $\Sigma$ via $(x\cdot \phi)(p,c) := \phi(x^{-1}p, x^{-1}c)$.
\end{definition}

\begin{definition}[Wall consistent orientations]
	An orientation $\phi$ of $\Sigma$ is \emph{wall consistent} if for any wall $H$ in $\Sigma$ and all alcoves $c,d$ which are in a same half-space of $H$ and have panels $p$ and $q$ in $H$ one has: $\phi(p,c)=\phi(q,d)$.   
	We may then call a half-space $H^\varepsilon$ of $H$, such that $\phi(p,c)=+1$ for one (and hence every) adjacent alcove in $H^\varepsilon$, a \emph {positive side} of $H$ with respect to $\phi$ or simply \emph{$\phi$-positive side}. The \emph{$\phi$-negative sides} are defined analogously.
\end{definition}


There are several ways to define a natural orientation on a Coxeter complex. We first introduce one class of orientations which works for arbitrary Coxeter groups. They are induced by a choice of an alcove or, equivalently, a regular point in a (geometric realization of) a Coxeter complex and are hence called alcove (or regular) orientations. 

\begin{definition}[Alcove orientation] \label{def:alcove orientation}
	Let $c$ be a fixed alcove in $\Sigma$. For any alcove $d$ and any panel $p$ in $d$, let $\phi_{c}(p,d)$ be $+1$ if and only if $d$ and $c$ lie on the same side of the wall spanned by $p$. The resulting orientation $\phi_{c}$ is called the \emph{alcove  orientation towards $c$} or short the \emph{$c$--orientation}. 
\end{definition}

Similar to Definition~\ref{def:alcove orientation} but more generally one can define an orientation with respect to a choice of any simplex, or in fact any point in a geometric realization of $\Sigma$. Obviously the alcove orientations are a sub-class of the orientations introduced in the next definition.  

\begin{definition}[Simplex orientation] \label{def:simplex orientation}
	Let $b$ be any simplex in $\Sigma$. For any alcove $c$ and any panel $p$ incident to $c$, let $\phi_b(p,c)$ be $+1$ if and only if either $c$ and $b$ lie on the same side of the wall $H$ containing $p$, or if $b$ lies inside $H$. The resulting orientation $\phi_b$ is called the \emph{simplex orientation towards $b$} or short the \emph{$b$--orientation}.  
\end{definition}

\begin{example}[Alcove and simplex orientation]
	Figure~\ref{fig:simplex orientations} shows two different simplex orientations on a type $A_2$ Coxeter complex. The one on the left hand side is induced by the alcove labeled $c$, while the one on the right hand side is induced by the panel $p$. 
\end{example}

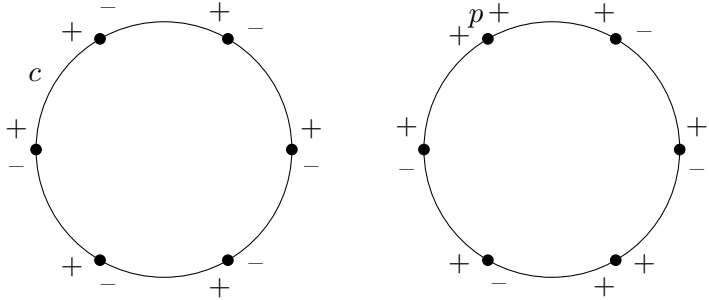
\begin{figure}[h]
	\begin{minipage}[c]{0.4\textwidth}
		\begin{tikzpicture}
		\def \n {6}
		\def \radius {1.7cm}
		\def \margin {8} 
		
		\draw circle (\radius);
		\node[draw=none, fill=none, inner sep=0pt, minimum width=6pt] at ({360/\n * 2.5}:{\radius+\margin}) {$c$};

		\foreach \s in {1,...,\n}
		{ 
			\node[circle, draw, fill=black!100, inner sep=0pt, minimum width=4pt] at ({360/\n * (\s - 1)}:\radius) {};
		}
		
		\foreach \s in {1,...,3}
		{ 
			\node[draw=none, fill=none] at ({360/\n * (\s - 1)+\margin}:{\radius+\margin}) {+};
			\node[draw=none, fill=none] at ({360/\n * (3+\s - 1)-\margin}:{\radius+\margin}) {+};
			
			\node[draw=none, fill=none] at ({360/\n * (\s - 1)-\margin}:{\radius+\margin}) {--};
			\node[draw=none, fill=none] at ({360/\n * (3+\s - 1)+\margin}:{\radius+\margin}) {--};
		}
		\end{tikzpicture}	
	\end{minipage}
	\begin{minipage}[c]{0.4\textwidth}
		\begin{tikzpicture}
		\def \n {6}
		\def \radius {1.7cm}
		\def \margin {9} 
		
		\draw circle (\radius);
		\node[draw=none, fill=none, inner sep=0pt, minimum width=6pt] at ({360/\n * 2}:{\radius+\margin}) {$p$};

		\foreach \s in {1,...,\n}
		{ 
			\node[circle, draw, fill=black!100, inner sep=0pt, minimum width=4pt] at ({360/\n * (\s - 1)}:\radius) {};
		}
		
		\foreach \s in {1,...,3}
		{ 
			\node[draw=none, fill=none] at ({360/\n * (\s - 1)+\margin}:{\radius+0.8*\margin}) {+};
			\node[draw=none, fill=none] at ({360/\n * (3+\s - 1)-\margin}:{\radius+0.8*\margin}) {+};
		}
		
		\node[draw=none, fill=none] at ({360/\n * (2)-\margin}:{\radius+0.8*\margin}) {+};
		\node[draw=none, fill=none] at ({360/\n * (5)+\margin}:{\radius+0.8*\margin}) {+};
		
		\node[draw=none, fill=none] at ({360/\n * (1)-\margin}:{\radius+0.8*\margin}) {--};
		\node[draw=none, fill=none] at ({360/\n * (6)-\margin}:{\radius+0.8*\margin}) {--};
		\node[draw=none, fill=none] at ({360/\n * (3)+\margin}:{\radius+0.8*\margin}) {--};
		\node[draw=none, fill=none] at ({360/\n * (4)+\margin}:{\radius+0.8*\margin}) {--};
		
		\end{tikzpicture}	
	\end{minipage}
	\caption{An alcove (left) and panel orientation (right) on the type $A_2$ Coxeter complex. }
	\label{fig:simplex orientations}
\end{figure}

\begin{lemma}[Basic properties]\label{lem:basic properties orientations}
	Let $(W,S)$ be a Coxeter group with Coxeter complex $\Sigma$. The following are true: 
	\begin{enumerate}[label=(\roman*)]
		\item\label{item:simplex} All simplex orientations are wall consistent and locally nonnegative. 
		\item All alcove orientations are wall consistent and locally nontrivial. 
	\end{enumerate}
\end{lemma}
\begin{proof}
	To see \ref{item:simplex} observe that for any wall $H$ there are two cases for the defining simplex $b$ of the given simplex orientation $\phi=\phi_b$. Either $b$ is contained in $H$ in which case both sides of $H$ are positive, or $b$ is contained in exactly one of the two sides of $H$ making this the positive side. In any case, two alcoves on a same side of $H$ with panels in $H$ always obtain the same sign under the given orientation $\phi$. Hence $\phi$ is wall consistent. From what we have said it is also clear that a simplex orientation can not assign $-1$ simultaneously to two alcoves sharing a panel. This implies \ref{item:simplex}. To deduce the second item it is enough to see that in this case there is no wall with two positive sides.      
\end{proof}

\subsection{Orientations on affine Coxeter complexes}
We now restrict to the affine case and introduce the class of orientations we will study most in this paper. It is determined by a choice of a chamber at infinity.

A wall consistent orientation chooses the same sign for all chambers having a panel in the same hyperplane $H$ and that are on the same side of $H$. This amounts to choosing a \emph{positive side} of $H$.  However, there is no need to choose  the positive sides of the  hyperplanes in a consistent way. But if so we will call these orientations periodic. See the next definition.   

\begin{definition}[Periodic orientations]
	A wall-consistent orientation $\phi$ of an affine Coxeter complex is \emph{periodic} if for any two parallel hyperplanes $H_1$ and $H_2$ and corresponding half-spaces $H_1^{\varepsilon_1}$ and $H_2^{\varepsilon_2}$, if $H_1^{\varepsilon_1} \subset H_2^{\varepsilon_2}$ then $H_1^{\varepsilon_1}$ is $\phi$-positive if and only if $H_2^{\varepsilon_2}$ is $\phi$-positive.
\end{definition}

Obviously the trivial orientations on an affine Coxeter complex are periodic. Note that the simplex induced orientations are not periodic as in every parallel class of hyperplanes one can find representatives having the defining simplex on different sides.

Periodic orientations have the nice property that they naturally induce an orientation on the boundary. 
We had already studied this interplay in Section 3 of \cite{MST}. Compare in particular Definitions 3.5 and 3.7 as well as Lemma 3.6. in \cite{MST} where one can essentially find what we recollect in \ref{lem:spherical inherited}, \ref{lem:affine inherited} and \ref{def:Weyl chamber orientation} below.  

\begin{lemma}[Induced spherical orientations]\label{lem:spherical inherited}
	Any periodic orientation $\phi$ on an affine Coxeter complex $\Sigma$ induces a wall-consistent orientation $\partial\phi$ on the spherical complex $\partial\Sigma$. 
	We will call $\partial\phi$ the \emph{orientation (at infinity) induced by $\phi$}.  
	In case $\phi$ is locally non-negative or non-trivial, then so is $\partial\phi$. 
\end{lemma}
\begin{proof}
	Let $M$ be a wall in $\partial \Sigma$, that is a parallel class of walls in $\Sigma$, and let $a$ be a chamber in $\partial\Sigma$ having a panel $p$ in $M$. Then there exists a Weyl chamber $C_a$ in $\Sigma$ representing $a$ which has a bounding wall $H_M$ in the parallel class $M$. Denote by $c$ the tip of $C_a$, that is the alcove in $C_a$ which contains the cone point of the Weyl chamber $C_a$. Then $c$ is, by construction, an alcove in $\Sigma$ with a panel $q$ in $H_M$. Now we can put $\partial\phi(a,p)\define\phi(c,q)$. As $\phi$ is periodic this definition does not depend on the choice of $C_a$ and $\partial\phi$ is automatically wall consistent as well.   
	It is not hard to see that the properties  locally non-negative or non-trivial will also be satisfied for the induced orientation.  
\end{proof}

The converse is also true. 

\begin{lemma}[Induced affine orientations]\label{lem:affine inherited}
	For a given affine Coxeter complex $\Sigma$ let $\phi$ be a wall-consistent orientation of $\Delta\define \partial\Sigma$.
	Then there exists a unique periodic orientation $\widetilde\phi$ of $\Sigma$ such that $\partial\widetilde\phi=\phi$. 
	We will call $\widetilde\phi$ the \emph{(affine) orientation induced by $\phi$}. 
	In case $\phi$ is locally non-negative or non-trivial, then so is $\widetilde\phi$. 
\end{lemma}
\begin{proof}
	For a hyperplane $H$ in $\Sigma$ we  choose a side $H^\varepsilon$ to be positive, respectively negative, if $\partial H^\varepsilon$ is a positive, respectively negative, side of the hyperplane $\partial H$ in $\Delta$. This uniquely determines  $\widetilde\phi$. 
\end{proof}

Alcove orientations on a spherical Coxeter complex $\Delta$ are wall consistent and locally non-trivial by Lemma~\ref{lem:basic properties orientations}. Hence they induce orientations on affine Coxeter complexes with $\Delta$ as their boundary by Lemma~\ref{lem:affine inherited}. One can view these as orientations on an affine $\Sigma$ determined by alcoves in the boundary $\partial\Sigma=\Delta$. We summarize this special case of induced affine orientations in the following definition. 

\begin{definition}[Weyl chamber orientations]\label{def:Weyl chamber orientation}
	Suppose $\Sigma$ is an affine Coxeter complex with boundary $\Delta$ and let $\sigma \in \Delta$ be some chamber. 
	Then the \emph{Weyl chamber orientation with respect to $\sigma$} (or short the \emph{$\sigma$--orientation}) is the orientation $\widetilde\phi_\sigma$ on $\Sigma$ induced by the $\sigma$-simplex orientation $\phi_\sigma$. 
\end{definition}

\begin{remark}[Alternative description of Weyl chamber orientations] Note that one can also describe the Weyl chamber  orientation as follows. For any alcove $c$ and any panel $p$ in $c$, let $H$ be the affine wall containing $p$. The chamber $\sigma$ corresponds to an equivalence class of Weyl chambers in $\Sigma$. We may hence define $\phi_u(p,c)$ to be $+1$ if $\sigma$ has a representative $C_\sigma$ which lies on the same side of $H$ as $c$.  This is the viewpoint we had taken in \cite{MST}.
\end{remark}

\begin{remark}[More induced orientations]
	Links in a Coxeter complex are again Coxeter complexes. One can show that they inherit orientations from the orientations on the ambient space. We will not need this concept in the present paper and hence will not formally introduce it.
\end{remark}


%
%
\section{Folded galleries}\label{sec:folded galleries}
In this section we introduce positively folded galleries, discuss some of their properties as well as possible ways to construct other positively folded galleries from a given one.    We essentially follow the terminology of \cite{MST} which is slightly different from the one in \cite{GaussentLittelmann}, where the concept of a folded gallery was, to our knowledge, introduced first.


\subsection{General notions}

We start with the definition of a combinatorial alcove--to--alcove gallery. 

\begin{definition}[Combinatorial galleries]\label{def:CombGallery}
	A \emph{(combinatorial) gallery} in a Coxeter complex $\Sigma=\Sigma(W,S)$ is a sequence 
	\[
	\gamma=(c_0, p_1, c_1, p_2,  \dots , p_n, c_n), 
	\]
	of alcoves $c_i$ and panels $p_i$ where for all $i=1, \ldots, n$ the panel $p_i$ is contained in both $c_i$ and $c_{i-1}$.  
	The \emph{length} of $\gamma$ is defined to be $n+1$. We say that $\gamma$ is \emph{minimal} if there is no shorter gallery connecting the \emph{source} $c_0$ with the \emph{sink} $c_n$.      
\end{definition}

All of our combinatorial galleries will contain at least one alcove. It is easy to see that if $c_i \neq c_{i-1}$ there is no choice for the panel $p_i$. As combinatorial galleries are the only ones we work with in this paper we will skip the word `combinatorial' in most places. 

\begin{remark}[Other classes of galleries]
	Note that it also makes sense to define vertex-to-vertex,  vertex-to-alcove or simplex-to-simplex galleries. The differences in their behavior are quite subtle. Compare for example~\cite[{Section 3.2}]{MST} in particular Remark 3.13 there. 
	In addition one can allow for more general steps in the gallery, i.e. replace the alcoves $c_i$ in our definition by smaller dimensional simplices as done in \cite{GaussentLittelmann}. Again, the properties they have might differ from the ones discussed here and it is often quite technical to keep track of their differences. However, depending on the context it might be necessary to switch to a different class and/or study the relationships between two classes.  
\end{remark}	

\begin{definition}[Folds] A gallery $\gamma$ is said to be \emph{folded (or stammering)} if there is some $i$ such that $c_i = c_{i-1}$, and \emph{unfolded} (or non-stammering) otherwise. If for some $i$ the alcove $c_i=c_{i-1}$ we say $\gamma$ has a \emph{fold} at panel $p_i$ or position $i$. The set $\folds(\gamma)$ of folds in $\gamma$ is the set of all $1\leq i\leq n$ such that $\gamma$ has a fold at panel $p_i$.  
\end{definition}

\begin{example}[Illustrating (folded) galleries]\label{ex:folded galleries}
	When drawing pictures we typically illustrate a (folded) gallery by a continuous path in the Coxeter complex that walks through the chambers and panels in the gallery. The arrow points towards the sink of the gallery. A bend touching a panel of an alcove illustrates a fold at the respective panel and shows that the alcove is repeated in the gallery.\newline
	In Figure~\ref{fig:folded galleries} we show two galleries in a type $\tilde{A}_2$ Coxeter complex. The grey gallery walks from $a$ to $c$, is not folded and not minimal. The black gallery has source $a$ and sink $b$. The first bit of the gallery (up to panel $p_4$) agrees with the grey one. The black gallery has two folds at panels $p_4$ and $p_7$. 
\end{example}

\begin{figure}[hb]
	\begin{overpic}[width=0.5\textwidth]{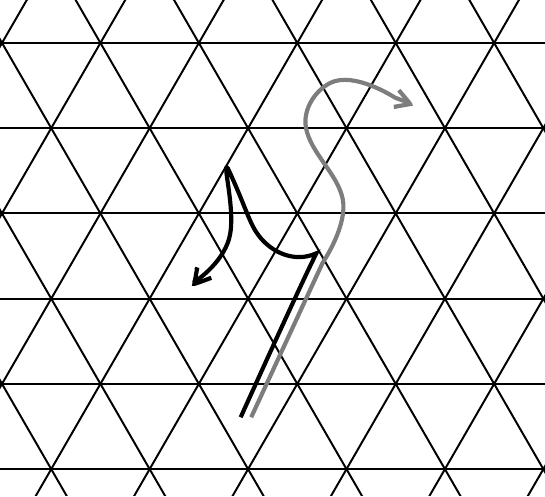}
		\put(45,10){$a$}
		\put(32,38){$b$}
		\put(68,70){$c$}
		\put(40,20){$p_1$}
		\put(45,30){$p_2$}
		\put(58,34){$p_3$}
		\put(62,40){$p_4$}
		\put(40,62){$p_7$}
	\end{overpic}
	
	\caption{This figure shows galleries in type $\tilde A_2$  with two folds (black) and no folds (gray). }	
	\label{fig:folded galleries}
\end{figure}
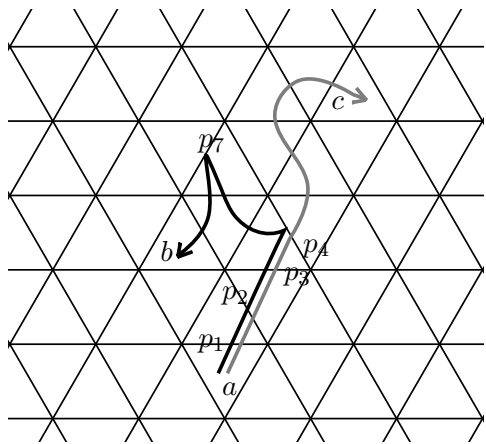

Taking orientations into account we can introduce the notion of a positively folded gallery. 

\begin{definition}[Positively folded galleries]\label{def:positively folded}
	A gallery $\gamma$  is \emph{positively} (respectively, \emph{negatively}) \emph{folded} with respect to an orientation $\phi$ if for all $1 \leq i \leq n$ either $c_{i-1}\neq c_i$, or $c_i=c_{i-1}$ and $\phi(p_i, c_i)=+1$. (respectively $-1$). 
\end{definition}

In other words, a gallery $\gamma$ has a positive fold at $c_i=c_{i-1}$ if the alcove $c_i$ is on the positive side of $p_i$. Analogously for negative folds the repeated alcove is on the negative side of the panel.    

\begin{remark}[Negative folds and opposite orientations]
	We will only be considering positively folded galleries as if some $\gamma$ is negatively folded with respect to an orientation $\phi$ then it is positively folded with respect to the \emph{opposite orientation} $-\phi$ defined by $-\phi(p,c) \define (-1)\cdot\phi(p,c)$.   
\end{remark}

Using the types of panels in a Coxeter complex we may associate a word to a combinatorial gallery.

\subsection{Galleries and words}

Fix a Coxeter system $(W,S)$ with Coxeter complex $\Sigma$. In this subsection we discuss the close relationship of galleries in $\Sigma$ and (decorated) words in $S$. 
By \emph{decorated words} we mean words in $S$ where we put hats on some of its letters. To make the wording easier words with no hats are also considered decorated words. If there are no hats on a (decorated) word we may also call it \emph{undecorated}.

\begin{definition}[Type of a gallery]\label{def:gallery type}
	Let $\gamma=(c_0,  p_1,  c_1, \dots, p_n, c_n)$ be a gallery. Its \emph{type}, denoted by $\type(\gamma)$, is the word in $S$ obtained as follows:
	\[
	\type(\gamma)\define s_{j_1}s_{j_2}\dots s_{j_n},
	\]
	where for $1 \leq i \leq n$ the panel $p_i$ of $\gamma$ has type $s_{j_i} \in S$.  
	We write $\Gamma_\phi^+(w)$ for the set of $\phi$-positively folded galleries of type $w$. 
	
	The \emph{decorated type}, denoted by $\dtype(\gamma)$ is the (decorated) word in $S$ obtained as follows:
	\[
	\dtype(\gamma)\define s_{j_1}\dots \hat{s}_{j_2} \dots s_{j_n},
	\]
	where the $s_{j_i} \in S$ are chosen as above and a hat is put on $s_{j_i}$ in case $c_{i-1}=c_i$ in $\gamma$.  
	By slight abuse of notation we call a letter with a hat a \emph{fold} of $\gamma$. 
	We write $\Gamma_\phi^+(\hat{w})$ for the set of positively folded galleries of decorated type $\hat w$. 
\end{definition}

\begin{lemma}[Galleries and words]\label{lem:words}
	Fix an alcove $c_0$ in a Coxeter complex $\Sigma=\Sigma(W,S)$. Then the following hold. 
	\begin{enumerate}[label=(\roman*)]
		\item\label{item:words} Words in $S$ are in bijection with the unfolded galleries with source $c_0$. 
		\item\label{item:decwords} The decorated words in $S$ are in bijection with the set of all galleries with source $c_0$ via $\dtype$. 
	\end{enumerate}	
\end{lemma}
\begin{proof}
	Note that in an unfolded gallery the alcove $c_i$ is obtained from $c_{i-1}$ by right-multiplication with the generator $s_{j_i}$. This implies \ref{item:words}.
	The fact that minimality is equivalent to the type being reduced was for example shown as Proposition 4.41 in \cite{AbramenkoBrown}. 
	
	To go from a decorated word $s_{j_1}\dots \hat{s}_{j_2} \dots s_{j_n}$ to a gallery define $c_i$ to be the $s_{j_i}$-neighbor of $c_{i-1}$ if there is no hat on $s_{j_i}$. In this case put $p_i\define c_i\cap c_{i-1}$. If there is a hat on $s_{j_i}$ put $c_i=c_{i-1}$ and choose as $p_i$ the unique panel of $c_i$ of type $s_{j_i}$. Hence item \ref{item:decwords}.
\end{proof} 

\begin{lemma}[Properties of galleries]
	For all galleries $\gamma$ the following hold.  
	\begin{enumerate}[label=(\roman*)]
		\item $\folds(\gamma)=\emptyset$ if and only if $\type(\gamma)=\dtype(\gamma)$.  
		\item $\gamma$ is minimal if and only if $\type(\gamma)$ is reduced and $\folds(\gamma)=\emptyset$.
	\end{enumerate}	
\end{lemma}

The notion of a footprint, defined below,  will allow us to characterize end-alcoves, i.e. sinks, of folded galleries. 

\begin{definition}[Footprint of a gallery]
	Let $\gamma =(c_0, p_1, c_1, \dots, p_n, c_n)$ be a combinatorial gallery of decorated type 
	$\dtype(\gamma)\define s_{j_1}\dots \hat{s}_{j_2} \dots s_{j_n}$. 
	The \emph{footprint} $\foot(\gamma)$ of $\gamma$ is the gallery obtained by deleting all the pairs $p_i, c_i$ for which the letter $s_i$ in $\dtype(\gamma)$ carries a hat.    	
\end{definition}

\begin{example}[Footprints]\label{ex:footprint}
	Note that the footprint of a given folded gallery $\gamma$ is shorter than $\gamma$ and unfolded (by construction) but need not be minimal. On the right hand side of Figure~\ref{fig:footprint} the black gallery with source $a$ and sink $d$ has as its footprint the minimal dashed grey gallery from $a$ to $d$. Here the panel $p_4$ and the chamber adjacent to it got deleted. The black gallery with source $a$ and sink $b$ on the left has a non-minimal footprint.  Both unfolded galleries, shown dotted, are minimal with source $a$ and sink $e$, respectively $c$. 
\end{example}

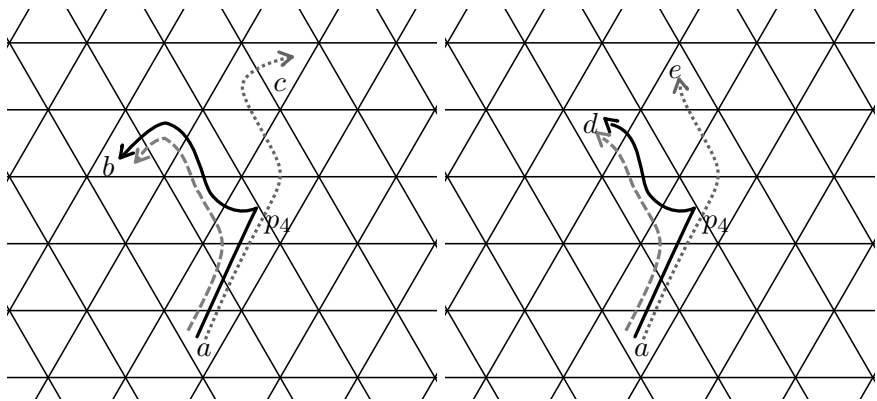
\begin{figure}[hb]
	\begin{overpic}[width=0.45\textwidth]{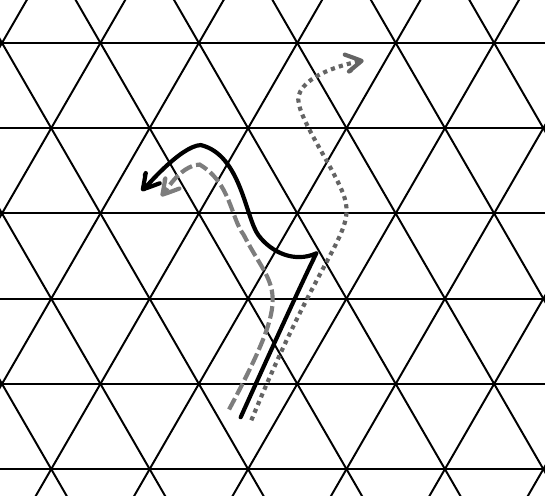}
		\put(44,10){$a$}
		\put(22,52){$b$}
		\put(62,72){$c$}
		\put(60,40){$p_4$}
	\end{overpic}
	\begin{overpic}[width=0.45\textwidth]{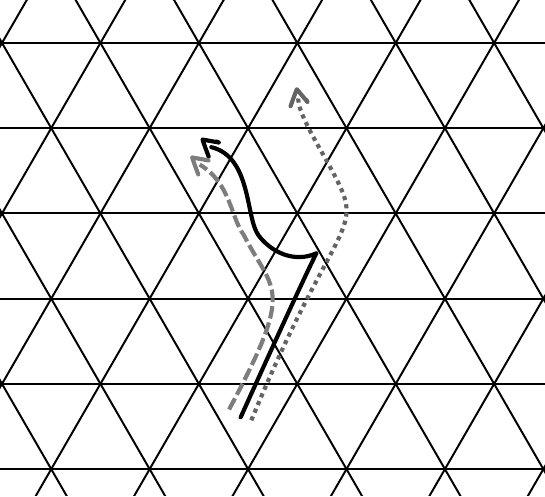}
		\put(44,10){$a$}
		\put(32,62){$d$}
		\put(52,75){$e$}
		\put(60,40){$p_4$}
	\end{overpic}
	
	\caption{This figure shows galleries (black), their unfolded images (dotted grey) and footprints (dashed grey). }	
	\label{fig:footprint}
\end{figure}

From the right-action of the Coxeter group $W$ on $\Sigma$ one obtains that the type of the footprint is a word such that the element it defines corresponds to the end-alcove of a folded gallery. 

\begin{lemma}[Footprint and end-alcoves]\label{lem:footprint final alcove}
	The final alcove of any combinatorial gallery $\gamma =(c_0, p_1, c_1, \dots, p_n, c_n)$ can be computed using the type of its footprint, namely $c_n=c_0\cdot w, \text{ where }w=\type(\foot(\gamma)).$ 
\end{lemma}
\begin{proof}
	In the footprint of a gallery all the folds are deleted. That is, in the footprint $\foot(\gamma) =(c_0, q_1, d_1, \dots, q_m, d_m)$, where $m=n-\#\folds(\gamma)$, every alcove $d_i$ is obtained from $d_{i-1}$ via right-multiplication with $s_i\define{\type(q_i)}$. Hence the claim of the lemma.  
\end{proof}

\subsection{Modification of galleries}\label{sec:modifications}

There are several ways to manipulate a positively folded gallery. In \cite{MST} we have made crucial use of the Littelmann root operators from \cite{Littelmann1} which were defined for galleries in \cite{GaussentLittelmann}. In Sections 6, 8.1, 8.3 and 9 of \cite{MST} we moreover introduced several methods to explicitly construct and manipulate galleries via extensions, conjugation or  concatenation. Ram \cite{Ram} as well as Parkinson, Ram and C. Schwer \cite{PRS} also discussed concatenations of folded galleries dressed as alcove walks. Kapovich and Millson studied the closely related Hecke paths and ways to construct them in \cite{KM}. 

In this subsection we discuss two kinds of manipulations of galleries: the natural action of $W$ and explicit folding and unfolding. In addition we introduce an equivalence relation on folded galleries induced by braid moves on the type. 

\begin{notation}[W-action on galleries] 
	It is clear from the definition of galleries and from the natural left-action of $W$ on $\Sigma$ that the Coxeter group $W$ also acts from the left on the set of all galleries in $\Sigma$. Write $x.\gamma$ for the image $(x.c_0, x.p_1, x.c_1, \dots, x.p_n, x.c_n)$ of $\gamma=(c_0, p_1, c_1, \dots, p_n, c_n)$ under $x\in W$. 
\end{notation}

Let us record a key property of this action in the following lemma.     

\begin{lemma}[$W$-action on positively folded galleries]\label{lem:left-action} 
	Let $(W,S)$ be an affine Coxeter system with Coxeter complex $\Sigma$ and choose a chamber $a$ in $\partial\Sigma$. A gallery $\gamma$ is $\phi_a$--positively folded  if and only if $x.\gamma$ is $\phi_{x.a}$ positively folded. 
	Here $x.a$ is the equivalence class of the Weyl chamber $x.C_a$ for any representative $C_a \subset \Sigma$ of $a$.  
\end{lemma}
\begin{proof}
	The group $W$ acts by isometries on $\Sigma$. This implies that galleries are mapped to galleries and that the action preserves the number and positions of folds. To see the rest check that an alcove $c$ is on the $\phi_a$-positive side of a hyperplane $H$ if and only if $x.c$ is on the $\phi_{x.a}$-positive side of $x.H$. 
\end{proof}

We will now introduce explicit foldings of galleries along panels. 

\begin{definition}[(Un-)foldings of galleries]\label{def:folds}
	Let $\gamma=(c_0, p_1, c_1, \dots, p_n, c_n)$ be a gallery and write $r_i$ for the reflection across the hyperplane $H_i$ containing the panel $p_i$. 
	For any $i\in\{1,2,\ldots,n\}$ define 
	\[
	\gamma^i\define(c_0, p_1, c_1, \dots, p_i, r_i c_i, r_i p_{i+1}, r_i c_{i+1}, \dots,  r_i p_n, r_i c_n).  
	\] 
	We call $\gamma^i$ a \emph{(un-)folding} of $\gamma$ at panel $i$, depending on whether $\gamma$ was  folded or not at $i$. 
\end{definition}

The next lemma follows from the fact that reflections are type preserving. 

\begin{lemma}[Elementary properties of folds]\label{lem:properties of folds} 
	Every (un-)folding $\gamma^i$ of a gallery $\gamma$ is again a gallery of the same type as $\gamma$, that is $\type(\gamma)=\type(\gamma^i)$. The number of folds decreases by one for an unfolding and increases by one for a folding. Moreover, $(\gamma^i)^i=\gamma$.  
\end{lemma}	

\begin{lemma}[Folds commute]\label{lem:folds commute}
	Let $\gamma=(c_0, p_1, c_1, \dots, p_n, c_n)$ be a gallery and $1\leq i,j \leq n$. Then $(\gamma^i)^j=(\gamma^j)^i$. 
\end{lemma}
\begin{proof}
	Lemma~\ref{lem:properties of folds} deals with the case that $i=j$. So assume without loss of generality that $i<j$. 
	Then 
	\[
	(\gamma^j)^i=(c_0, p_1, c_1, \ldots, c_{i-1}, p_i, rc_i, \ldots, rc_{j-1}, rp_j, rr'c_j, \ldots, rr'c_n) 
	\] 
	where $r$ is the reflection along the hyperplane spanned by $p_i$ and $r'$ the reflection on the hyperplane spanned by $p_j$. 
	And 
	\[
	(\gamma^i)^j=(c_0, p_1, c_1, \ldots, c_{i-1}, p_i, rc_i, \ldots, rc_{j-1}, rp_j, r''rc_j, \ldots, r''rc_n) 
	\] 
	where $r$ is as above and $r''$ the reflection on the hyperplane spanned by $rp_j$. 
	
	For every panel $p$ of an alcove $c$ the unique second alcove in $\Sigma$ containing $p$ is $c\type(p)$. Therefore the reflection along the hyperplane $H$ spanned by $p$ is the product $c\type(p)c^{-1}$. We obtain 
	\[
	r=c_{i-1}\tau(p_i)c_{i-1}^{-1}, \; r'=c_{j-1}\tau(p_j)c_{j-1}^{-1} \;\text{ and }\; r''=rc_{j-1}\tau(rp_j)c_{j-1}^{-1}r.
	\]
	Reflections preserve types. Therefore $\type(rp_j)=\type(p_j)$. It is now easy to check that $r'' r = rr'$ and hence $(\gamma^i)^j=(\gamma^j)^i$.  
\end{proof}

Because of Lemma~\ref{lem:folds commute} we can write $\gamma^{ij}$ in place of $(\gamma^i)^j$ and define folds with respect to subsets of the index set. 
Hence we can fold a gallery simultaneously at several panels which implies that \ref{def:multiple folds} below is well defined. 	

\begin{definition}[Multifoldings]\label{def:multiple folds}
	Let $\gamma=(c_0, p_1, c_1, \dots, p_n, c_n)$ be a gallery and let $I= \{i_1, i_2, \ldots, i_k\}$ be a subset of the index set $\{1,2,\ldots n\}$ of $\gamma$ such that $i_j \in\{1,2,\ldots,n\} \;\forall 1\leq j\leq k$. Then we define the \emph{multifolding} of $\gamma$ at $I$ to be the gallery $\gamma^I\define \gamma^{i_1i_2\cdots i_k}$ and call $I$ the \emph{multi-folding-index}.   
\end{definition}

Lemma~\ref{lem:properties of folds} and \ref{lem:folds commute} imply similar properties for multifoldings. 

\begin{lemma}[Properties of multifoldings]\label{lem:properties of multifolds}
	Let $\gamma$ be a gallery of length $n+1$ and $I\subset \{1,2,\ldots,n\}$. Then the following hold. 
	\begin{enumerate}[label=(\roman*)] 
		\item $\type(\gamma)=\type(\gamma^I)$, i.e. folding does not change the type. 
		\item $\folds(\gamma^I)=\folds(\gamma)\Delta I$, i.e. the set of folds of $\gamma^I$ is the symmetric difference of the folds of $\gamma$ with the multi-folding-index $I$. In particular $(\gamma^I)^J=\gamma^{I\Delta J}$ for all $J\subset\{1,2,\ldots,n\}$. 
	\end{enumerate}	
\end{lemma} 

From what we have discussed the following is immediate. 

\begin{corollary}[Unfolding]
	For every folded gallery $\gamma$ of type $w$ and length $n+1$ there exists a subset $I\subset\{1,2,\ldots,n\}$ such that $\gamma^I$ is unfolded and of the same type. 
\end{corollary}

In other words: Every folded gallery arises as a multifolding of an unfolded gallery of the same type. 

\begin{example}[Commuting folds and multifolds ]
	The gray gallery $\gamma^{4,7}$ in Figure~\ref{fig:multifolds} is the multifolding of the black gallery $\gamma$ at positions 4 and 7. 
	This figure also illustrates the fact that folds commute, which we have shown in Lemma~\ref{lem:folds commute}. The dotted and dashed galleries are the folds of $\gamma$ at positions 7, respectively 4. Both of them admit a fold that takes them to the gallery $\gamma^{4,7}$. 
\end{example}

\begin{figure}[htb]
	\begin{overpic}[width=0.5\textwidth]{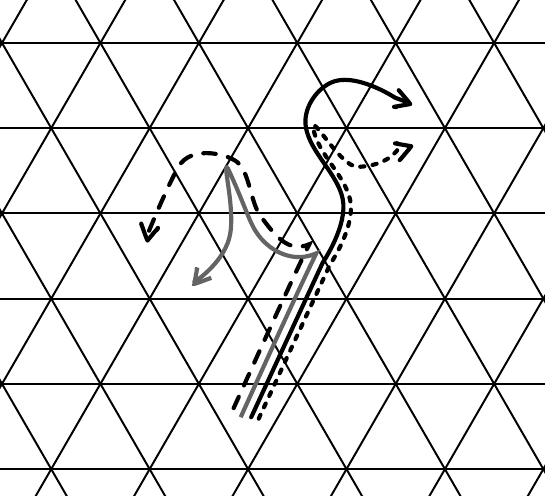}
		\put(30,32){$\gamma^{\{4,7\}}$}
		\put(25,42){$\gamma^4$}
		\put(77,60){$\gamma^7$}
		\put(77,72){$\gamma$}
	\end{overpic}
	
	\caption{This figure shows commuting folds at panels $4$ and $7$ of the black gallery $\gamma$. }	
	\label{fig:multifolds}
\end{figure}

\begin{notation}[Object A folds onto object B ]
	Let $(W,S)$ be a Coxeter system with Coxeter complex $\Sigma$ and suppose that $w$ is a word in $S$. Denote by $\gamma_w$ the unique unfolded gallery of type $w$ starting in $\id$. 
	We will write  
	\begin{enumerate}[label=(\roman*)]
		\item $\gamma \pfold{} \eta$ for galleries $\gamma, \eta$ if $\eta = \gamma^I$ for some $I\subset\{1,2,\ldots, n\}\setminus\folds(\gamma)$.
		\item $w\pfold{} u$ for a  word $u$ in $S$ if there exists a gallery $\eta$ with footprint $\foot(\eta)=u$ and  $\gamma_w\pfold{}\eta$ . 
		\item $w\pfold{} x$ for an element $x\in W$ if there exists a gallery $\eta$ with end-alcove $c_x$ such that $\gamma_w\pfold{}\eta$. 
	\end{enumerate}
	If in the above $\eta$ is positively folded with respect to some orientation $\phi$ we label the arrow with $\phi$ and write $A\pfold{\phi} B$. 
\end{notation}

\begin{notation}[Sets of  folded galleries]
	The set of all (multi-)folds of a gallery $\gamma$ is denoted by $\Gamma(\gamma)$. The set of all {(multi-)folds} of a gallery $\gamma$ that are positively folded with respect to a given orientation $\phi$ is denoted by $\Gamma_\phi^+(\gamma)$. We will sometimes write $\Gamma_\phi^+(w)$ for the set $\Gamma^+_\phi(\gamma_w)$, for $w$ a word and  $\gamma_w$ the unfolded gallery of type $w$.   
\end{notation}

\subsection{Statistics on positive folds}

In this subsection we restrict ourselves to Weyl chamber orientations on affine Coxeter complexes. So in the following we  assume, if not stated otherwise,  that $(\aW,S)$ is an affine Coxeter system with Coxeter complex $\Sigma$ and that $\phi=\widetilde\phi_a$ for some chamber $a\in\partial\Sigma$. 

The number of folds in a positively folded gallery with respect to a Weyl chamber orientation has natural bounds. The formula in Proposition~\ref{prop:bounds on folds} says that the length of the longest element in the associated spherical Weyl group is a uniform upper bound while \emph{reflection length} $\ell_R$, that is the length of an element measured with respect to the larger generating set $R$ of all reflections of $\aW$, provides a lower bound. 

\begin{prop}[Bounds on the number of folds]\label{prop:bounds on folds}
	Let $w_0$ denote the longest element in $\sW$. For any $x\in\aW$  and any $\phi$-positive multifolding $\gamma$ of a minimal gallery $\gamma_x$ with sink the alcove labeled by $x$ one has: 
	\[
	\ell_R(xy^{-1}) \leq \vert \folds(\gamma)\vert  \leq \ell(w_0), \text{ with } y\define\tau(\foot(\gamma)).
	\] 
\end{prop}
\begin{proof}
	By Lemma~\ref{lem:footprint final alcove} the element $y=\tau(\foot(\gamma))$ corresponds to the final alcove of $\gamma$. With this observation the claim directly follows from Corollary 4.25 and Lemma 4.26 in \cite{MST}. 
\end{proof}

Note that Section 4 of \cite{MST} contains a more detailed study of folds, crossings and dimensions of galleries. 

We now introduce a valuation on elements of $W$, respectively the corresponding alcoves in $\Sigma$. We have not worked out the precise connection, but this seems closely related to the notion of load-bearing walls introduced in \cite{GaussentLittelmann}.

\begin{notation}[Separating hyperplanes] \label{not:separating hyperplanes}
	Denote by $\hyp(\Sigma)$ the collection of all hyperplanes in a Coxeter complex $\Sigma$. 
	For some fixed alcove $c$ in $\Sigma$ let $\hyp(c)$ be the set of hyperplanes separating $c$ and the identity alcove $\id$. Then  $\hyp(c)=\hyp^+_\phi(c)\sqcup\hyp^-_\phi(c)$, where $\hyp^+_\phi(c)$ is the subset of $\hyp(c)$ for which $c$ is on a positive side and $\hyp^-_\phi(c)$ the ones for which $c$ is on a negative side. 
\end{notation}

In the following we will write $\Ch(\Sigma)$ for the collection of all alcoves in a Coxeter complex $\Sigma$.

\begin{definition}[$\phi$-valuation]\label{def:valuation}
	We define the \emph{$\phi$-valuation} to be the map 
	\[
	\v_\phi:\Ch(\Sigma) \to \Z \;\text{ with }\; c\mapsto \v_\phi(c)\define\vert\hyp^+_\phi(c) \vert - \vert\hyp^-_\phi(c)\vert.  
	\]
\end{definition}

The function introduced in the next definition can be thought of as  extension of a wall consistent orientation to pairs of alcoves and hyperplanes. It decides whether a given alcove is on a positive side of a hyperplane.

\begin{definition}[Positive sides of hyperplanes]
	We define a function $p_\phi$ on $\Ch(\Sigma)\times\hyp(\Sigma)$ as follows: 
	\[
	p_\phi(c,H)\define \left\{ \begin{array}{ll} 1 & \text{ if $c$ is on a $\phi$-positive side of $H$ } \\ 0 & \text{ otherwise. } \end{array} \right.
	\]
\end{definition}	

\begin{lemma}[Formulas for $\v_\phi$]
	Denote the identity alcove in $\Sigma$ by $\id$. Then, for all $c\in\Ch(\Sigma)$ we have  
	\[
	\v_\phi(c)=\sum_{H\in \hyp(\Sigma)} (p_\phi(c,H) - p_\phi(1,H)). 
	\]
\end{lemma}
\begin{proof}	
	Recall from \ref{not:separating hyperplanes} that the set $\hyp(c)$  of hyperplanes separating $c$ from $\id$ can be written as a disjoint union $\hyp(c)=\hyp^+_\phi(c)\sqcup\hyp^-_\phi(c)$. 
	Now every hyperplane $H$ has a positive and a negative side and hence 
	either 
	$(p_\phi(c,H) = p_\phi(1,H))$ or they are different, $H\in\hyp(c)$ and  $(p_\phi(c,H)$ and $ p_\phi(1,H))$ differ by  $\pm 1$.  Therefore 
	\[
	\v_\phi(c)=\sum_{H\in \hyp(\Sigma)} (p_\phi(c,H) - p_\phi(1,H)) = \sum_{H\in \hyp(c)} (p_\phi(c,H) - p_\phi(1,H))
	\]
	and the result follows from combining summands. 
\end{proof}

\begin{lemma}[Length and valuations]\label{lem:length}
	Fix $x\in W$ and denote by $c_x$ the alcove corresponding to $x$. Then 
	\[
	\ell(x)\geq \v_\phi(c_x).
	\]
\end{lemma}
\begin{proof}
	Simply compute $\ell(x)=\vert \hyp(c_x)\vert = \vert \hyp^+_\phi(c_x)\vert + \vert \hyp^-_\phi(c_x)\vert \geq \v_\phi(c_x)$. 
\end{proof}

\begin{definition}[$\phi$-dominant alcoves]\label{def:dominant}
	An alcove $c$ is \emph{dominant} with respect to an orientation $\phi$ if $\v_\phi(c)=\ell(c)$.  	
\end{definition}

Recall that for a given chamber $a$ in $\partial\Sigma$ we  write $\widetilde \phi_a$ for the Weyl chamber orientation on $\Sigma$  induced by the simplex orientation $\phi_a$ on $\partial\Sigma$.  By slight abuse of notation we will write $a\in\sW$. 

\begin{lemma}[length via Weyl chamber orientations]\label{lem:dominant}
	For every $x\in\aW$ and its corresponding alcove $c_x$ we have  
	\[
	\ell(x) = \max_{a\in\sW} \v_{\tilde\phi_a}(c_x).
	\]
\end{lemma}
\begin{proof}
	Let $C$ be the unique Weyl chamber with tip at the origin $0$ containing the alcove $c_x$ and write $a\define\partial C$ for the chamber at infinity determined by $C$. Let $\phi$ be the affine inherited valuation from the alcove orientation towards $a$ at infinity. Then any mininal gallery from $\id$ to $c_x$ has the property that its panels span hyperplanes for which $c_x$ is on the $\phi_a$-positive side. Therefore 
	\[
	\ell(x)=\vert\hyp_{\phi}^+(c_x) \vert =\v_{\phi}(c_x)\leq \max_{a\in\sW} \v_{\tilde\phi_a}(c_x).  
	\]
	The statement now follows from Lemma~\ref{lem:length}. 	
\end{proof}

\begin{remark}[$\phi$--dominant alcoves]
	In view of Definition~\ref{def:dominant} the assertion of  Lemma~\ref{lem:dominant} says that for every alcove $c$ there exists a Weyl chamber orientation $\phi$ such that $c$ is dominant with respect to $\phi$. So the lemma should not be surprising. 
\end{remark}

\begin{lemma}[Reflections increasing $\v$]\label{lem:reflection-v}
	Let $\varphi\in\Dir(\aW)$ be a direction and let $r\in R$ be a reflection in $\aW$ across a hyperplane $H_r$. 
	Then for any $x\in\aW$, $\v_\varphi(x) > \v_\varphi(rx)$ if and only if $x$ is on the $\varphi$-positive side of $H_r$.   
\end{lemma}
\begin{proof}
	
	It suffices to show one implication of the equivalence, since the other implication is obtained by exchanging $x$ and $rx$, and equality of $\v_\varphi(x)$ and $\v_\varphi(rx)$ is impossible by parity. So let $x$ lie on the $\varphi$-positive side of $H_r$.
	
	Consider the set $\shyp$ of those hyperplanes separating $x$ and $rx$. Let $\shyp^+$ be the set of hyperplanes $H \in \shyp$ such that $x$ is on the $\varphi$-positive side of $H$ and $rx$ is on the $\varphi$-negative side. Moreover, let $\shyp^- = \shyp \setminus \shyp^+$ be the set of hyperplanes $H \in \shyp$ such that $rx$ is on the $\varphi$-positive side of $H$ and $x$ is on the $\varphi$-negative side.
	
	Observe that $\v_\varphi(x) - \v_\varphi(rx) = |\hyp_\varphi^+(x)| - |\hyp_\varphi^+(rx)| = |\shyp^+| - |\shyp^-|$. Therefore it suffices to show that $|\shyp^+| > |\shyp^-|$.
	
	Observe also that the map $H \mapsto rH$ is an involution on $\shyp$ with exactly one fixed point $H_r$, where $H_r$ is the reflection hyperplane of $r$. We claim that $\shyp^- \cap r\shyp^- = \emptyset$. If this is true, then $r\shyp^-$ is a proper subset of $\shyp^+$ (proper because $H_r = rH_r$ lies in $\shyp^+$ but not in $r\shyp^-$), so $|\shyp^+| > |r\shyp^-| = |\shyp^-|$ and the proof is done.
	
	We now want to prove the claim. For any $H \in \shyp$, denote by $H^+$ and $H^-$ the half-spaces of $H$ on the $\varphi$-positive and $\varphi$-negative side respectively.
	
	Assume for contradiction that there is some $H \in \shyp^- \cap r\shyp^-$. Let $J$ be the intersection of  $H^+$ with  $(rH)^+$. The set $J$ is nonempty, since $rx$ lies in $J$, and its boundary $\partial J = \partial H^+ \cap \partial (rH)^+$ at infinity contains $\sigma$.
	
	Now $r((rH)^+)$ is some half-space of $H$ that contains $rrx = x$, so it must be the $\varphi$-negative half-space $H^-$ because $H \in \shyp^-$.

	Now $rJ = r(H^+) \cap r((rH)^+) \subseteq r((rH)^+) = H^-$, and $J \subseteq H^+$, so $J$ and $rJ$ are disjoint sets. Hence $J$ cannot contain a fixed point of $r$, so by convexity must be contained in a single half-space of $H_r$. As the boundary of $J$ contains $\sigma$, we find that $J \subset H_r^+$. Since $rx$ lies in $J$, we find that $rx$ lies on the $\varphi$-positive side of $H_r$, so $x$ lies on the $\varphi$-negative side of $H_r$, which contradicts our choice of $x$. This proves the claim.
\end{proof}

\section{Braid invariant orientations}

We introduce the notion of a braid invariant orientation in this section. It will later be used to prove that certain shadows do not depend on a chosen word representing a given element in a Coxeter group. 


\begin{remark}[Braid moves on words] 
	Ideally one would define an equivalence relation on galleries coming from braid moves on words. 
	The word property, discovered by Matsumoto \cite{Matsumoto} and Tits \cite{Tits} in the 1960s, implies that any two reduced expressions for an element $x\in W$ can be connected via a sequence of braid moves. (For a textbook reference see Theorem 3.3.1 in \cite{BjoernerBrenti}.) 
	A braid move can also be considered for a folded gallery $\gamma$ by changing the sub-gallery corresponding to the word $\underbrace{stst\ldots}_{m_{st}}$ to the sub-gallery of type $\underbrace{tsts\ldots}_{m_{st}}$ while keeping the folds in the same positions, i.e. on the letters with the same index in the word. This however, will in general not be well defined, as the new sub-gallery may end in a different alcove.
\end{remark}

\begin{definition}[Braid invariant orientations]\label{def:braid invariance}
	Let $\Sigma$ be a Coxeter complex for the Coxeter system $(W, S)$.
	An orientation $\phi$ on $\Sigma$ is \emph{braid invariant} if for any braid equivalent words $w, w'$ in $S$ and any $x \in W$, it is true that $w \pfold{\phi} x$ if and only if $w' \pfold{\phi} x$. 
	We call $\phi$ \emph{strongly braid invariant} if and only if $y\phi$ is braid invariant for all $y \in W$.
\end{definition}

\begin{notation}[Folds for braid invariant orientations]
	Let $\Sigma$ be a Coxeter complex for the Coxeter system $(W, S)$, let $\phi$ be a braid invariant orientation on $\Sigma$. Given two elements $x, y \in W$, we define $x \pfold{\phi} y$ to be equivalent to $w \pfold{\phi} y$ for any (and thus every) reduced expression $w$ of $x$.
\end{notation}

It is obvious that the trivial positive/negative orientation is (strongly) braid invariant. 
Proposition 4.33 of \cite{MST} implies that the Weyl chamber orientations are braid invariant. We include an elementary proof for this fact in Proposition~\ref{prop:weyl = braid inv} below.

\begin{prop}[Weyl chamber orientations are braid invariant]\label{prop:weyl = braid inv}
	Let $\Sigma$ be an affine Coxeter complex with boundary $\Delta$, and let $\tilde\phi_\sigma$ be a Weyl chamber orientation, induced by some chamber $\sigma$. Then $\Sigma$ is strongly braid invariant.
\end{prop}

Note that for any $x \in \aW$, we have $x\tilde\phi_\sigma = \tilde\phi_{x\sigma}$, thus strong braid invariance for all $\sigma$ follows immediately from braid invariance for all $\sigma$. For the proof of Proposition~\ref{prop:weyl = braid inv} we will need the following lemma.

\begin{lemma}[Braid invariant folds]\label{lemma:weyl = short braid inv}
	Suppose $(\aW,S)$ is a Coxeter system with Coxeter matrix $M=(m_{st})_{s,t\in S}$. Let $\tilde\phi_\sigma$ be a Weyl chamber orientation on $\Sigma$. Then for all words $w = \underbrace{stst\ldots}_{m_{st}}$, $w' = \underbrace{tsts\ldots}_{m_{st}}$ in $S$ and all $x \in \aW$, it is true that $w \pfold{\tilde\phi_\sigma} x$ if and only if  $w' \pfold{\tilde\phi_\sigma} x$.
\end{lemma}

\begin{proof}
	Since the type of the footprint of any folded gallery of type $w$ or $w'$ beginning at $\id$ can only contain symbols $s$ and $t$, the end of that gallery must lie in $\aW_{\{s,t\}}\id$. Therefore it suffices to consider only $x \in \aW_{\{s,t\}}$.
	
	Let $m = m_{st}$ and put $\gamma_w\define (c_0 = \id, p_1, \ldots, p_m, c_m)$. Let $\tilde c$ be the alcove in $\aW_{\{s,t\}}\id$ that lies closest to $\sigma$ i.e. it lies on the $\tilde\phi_\sigma$-positive sides of all walls that separate alcoves in $\aW_{\{s,t\}}\id$. Observe for any reflection $r \in \aW_{\{s,t\}}$ and any alcove $c \in \aW_{\{s,t\}}\id$ that $c$ lies on the positive side of $H_r$ if and only if the gallery distance (i.e. length of shortest connecting gallery) of $\tilde c$ to $c$ is smaller than the one to $rc$.

	\emph{Claim:} Let  $x$ be any element in $\aW_{\{s,t\}}$ and write $c_x$ for the alcove it represents. Then $w \pfold{\tilde\phi_\sigma} x$ if and only if  either $c_x = c_m$ or $\tilde c$ is has shorter gallery distance to $c_x$ than to $c_m$.
	
	Note that the right-hand side of the equivalence in this claim is symmetric in $s$ and $t$, so applying the claim twice immediately yields that $w \pfold{\tilde\phi_\sigma} x$ if and only if  $w' \pfold{\tilde\phi_\sigma} x$, as desired.
	
	Let us now prove the claim. For the case $c_x = c_m$ the gallery $\gamma_w$ immediately demonstrates $w \pfold{\tilde\phi_\sigma} x$, so we may suppose $c_x \neq c_m$ from now on.
	
	Suppose that $w \pfold{\tilde\phi_\sigma} x$, so there exists a folded gallery $\gamma = \gamma_w^I$ of type $w$ starting in $\id$ and ending in $c_x$. Let $I = \{i_1, \ldots, i_k\}$ for some indices $1 \leq i_1 < \ldots < i_k \leq m$. Note that $k > 0$, since $c_x \neq c_m$. Let $c_i^j$ or $p_i^j$ denote the $i$-th alcove or panel of the gallery $\gamma_w^{i_1\ldots i_j}$ for $j = 0, \ldots, k$. Note that any such gallery $\gamma_w^{i_1\ldots i_j}$ is $\tilde\phi_\sigma$-positively folded, since each folded panel of that gallery already lies at the same position as its corresponding folded panel in $\gamma_w^I$.
	
	Then for any such $j$, the alcove $c_{i_j-1}^{j-1} = c_{i_j-1}^{j} = c_{i_j-1}^{j}$ lies on the positive side of the hyperplane $H_j$ with respect to $\tilde\phi_\sigma$ containing $p_{i_j}^j$, by positivity of $\gamma_w^{i_1\ldots i_j}$. Now $(c_{i-1}^{j-1}, p_i^{j-1}, \ldots, c_m^{j-1})$ is an unfolded gallery of the same type as $(c_{i-1}, p_i, \ldots, c_m)$ and therefore minimal. This gallery starts on the $\tilde\phi_\sigma$-positive side of $H_j$ and passes through $H_j$. Therefore $c_m^{j-1}$ lies on the $\tilde\phi_\sigma$-negative side of $H_j$; and $c_m^j$, obtained from $c_m^{j-1}$ by reflection across $H_j$, is closer to $\tilde c$ than $c_m^{j-1}$ is.
	
	By induction over $j$, we find that $\tilde c$ is closer to $c_m^k = c_x$ than to $c_m^0 = c_m$. This proves one side of the claim.
	
	Suppose now that $\tilde c$ lies closer to $c_x$ than to $c_m$. We wish to find some $\tilde\phi_\sigma$-positive multifolding of $\gamma_w$ ending at $c_x$.
	
	\emph{Case 1:} $c_x$ and $c_m$ have different parity. Then there is a reflection $r \in \aW_{\{s,t\}}$ such that $rc_m = c_x$. Since $c_m$ and $\id$ lie on different sides of the hyperplane $H_r$, there is some $i$ such that $p_i$ lies on $H_r$.
	
	Then $c_x$ lies on the $\tilde\phi_\sigma$-positive side of $H_r$ and $c_m$ lies on the $\tilde\phi_\sigma$-negative side. Since $\gamma_w$ is minimal, this means that $c_{i-1}$ lies on the $\tilde\phi_\sigma$-positive side of $H_r$, so $\gamma_w^i$ is a $\tilde\phi_\sigma$-positively folded gallery of type $w$ from $\id$ to $rc_m = c_x$, demonstrating $w \pfold{\tilde\phi_\sigma} x$.
	
	\emph{Case 2:} $c_x$ and $c_m$ have the same parity, and $\tilde c = c_i$ for some $i = 0, \ldots, m$. We may assume $i \neq m$ because otherwise $\tilde c$ would lie closer to $c_x$ than to itself, which is not possible.
	
	Now the gallery $\gamma_w^{i+1}$ is positively folded and ends at $rc_m$ where $r$ is the reflection across the panel $p_{i+1}$. Since $p_{i+1}$ is adjacent to $\tilde c$, we find that the combinatorial distance between $\tilde c$ and $rc_m$ is exactly $1$ less than the distance between $\tilde c$ and $\tilde c_m$. Because of parity, $\tilde c$ still lies closer to $c_x$ than to $rc_m$.
	
	Since $c_x$ and $rc_m$ now have different parity, we find a reflection $r' \in \aW_{\{s,t\}}$ such that $r'rc_m = c_x$. Using our observation at the beginning of this proof, we find that the hyperplane $H$ corresponding to $r'$ now separates $\tilde c$ and $rc_m$. Since $(rc_{i+1} = \tilde c, rp_{i+2},\ldots, rc_m)$ is a minimal gallery from $\tilde c$ to $rc_m$, there exists some $j > i+1$ such that $rp_j$, the $j$-th panel of $\gamma_w^{i+1}$, lies in $H$. Therefore the gallery $(\gamma_w^{i+1})^j$ is the desired $\tilde\phi_\sigma$-positively folded gallery from $\id$ to $c_x$ of type $w$.
	
	\emph{Case 3:} $c_x$ and $c_m$ have the same parity, but $\tilde c \notin \{c_0, \ldots, c_m\}$. Then it must be the case that $c_0$ is closer to $\tilde c$ than $c_1$, so $\gamma_w^1 = (c_0, p_1, sc_1, sp_2, \ldots, sc_m)$ is $\tilde\phi_\sigma$-positively folded, and the alcoves of $\gamma_w^1$ contain all those alcoves in $\aW_{\{s,t\}}\id$ not yet covered by $\{c_0, \ldots, c_m\}$. Therefore $\tilde c = sc_i$ for some $1 < i \leq m$. Since $sc_m$ is adjacent to $c_m$, the alcove $\tilde c$ still lies closer to $c_x$ than $sc_m$, in particular this means that $i \neq m$.
	
	Since $sc_m$ and $c_x$ have different parity, we find some reflection $r \in \aW_{\{s,t\}}$ such that $r'sc_m = c_x$, and the hyperplane $H$ corresponding to $r$ separates $\tilde c$ and $sc_m$. So we find $j$ with $i < j \leq m$ such that the panel $sp_j$ lies in $H$. The gallery $(\gamma_w^1)^j$ is now the desired $\tilde\phi_\sigma$-positively folded gallery from $\id$ to $c_x$ of type $w$.
\end{proof}

We can now prove Proposition \ref{prop:weyl = braid inv}.
\begin{proof}
	Let $s, t \in S$ and write $w_{st} = \underbrace{stst\ldots}_{m_{st}}$ as well as $w_{ts} = \underbrace{tsts\ldots}_{m_{st}}$ for the two words making up the defining Coxeter relations.	
	Let $w = uw_{st}v$ and $w' = uw_{ts}v$ be any two words in $S$ differing by a braid move. Let $m = m_{st}$ and let $k$ and $l$ be the length of the subwords $u$ and $v$ respectively. Then $n := k + m + l$ is the length of $w$ and $w'$.
	
	Suppose that $w \pfold{\tilde\phi_\sigma} x$ for some $x \in \aW$. Then there exists a $\tilde\phi_\sigma$-positively folded gallery $\gamma = (c_0 = \id, p_1, \ldots, p_n, c_n = c_x)$ of type $w$. We now want to construct a $\tilde\phi_\sigma$-positively folded gallery $\gamma'$ of type $w'$ from $\id$ to $c_x$.
	
	Consider the subgallery $\gamma_1 := (c_k, p_{k+1}, \ldots, p_{k+m}, c_{k+m})$ of $\gamma$. This subgallery is a $\tilde\phi_\sigma$-positively folded gallery of type $w_{st}$. Choose $y, z \in W$ such that $c_k = c_y$ and $c_{k+m} = c_{yz}$. Then $\gamma_2 := y^{-1}\gamma_1$ is a $y^{-1}\tilde\phi_\sigma$-positively folded gallery of type $w_{st}$ from $\id$ to $y^{-1}c_{yz} = c_z$, which means that $w_{st} \pfold{y^{-1}\tilde\phi_\sigma} z$.
	
	Since $y^{-1}\tilde\phi_\sigma = \tilde\phi_{y^{-1}\sigma}$ is a Weyl chamber orientation, we can apply Lemma \ref{lemma:weyl = short braid inv} and find that $w_{ts} \pfold{y^{-1}\tilde\phi_\sigma} z$, so there exists some $y^{-1}\tilde\phi_\sigma$-positively folded gallery $\gamma'_2$ of type $w_{ts}$ from $\id$ to $c_z$.
	Multiplication with $y$ yields a gallery $\gamma'_1 = (c'_k = c_y = c_k, p'_{k+1}, \ldots, p'_{k+m}, c'_{k+m} = c_{yz} = c_{k+m}) := y\gamma'_2$ of type $w_{ts}$ from $c_k$ to $c_{k+m}$ that is $yy^{-1}\tilde\phi_\sigma$-positively folded.
	
	Now the gallery
	\[\gamma' := (c_0 = \id, p_1, \ldots, p_k, c_k = c'_k, p'_{k+1}, \ldots, p'_{k+m}, c'_{k+m} = c_{k+m},\ldots, p_n, c_n)\]
	constructed from $\gamma$ and $\gamma'_1$ is $\tilde\phi_\sigma$-positively folded from $\id$ to $c_x$, and the type of $\gamma'$ is exactly $uw_{ts}v = w'$ by construction. This shows $w' \pfold{\tilde\phi_\sigma} x$ as desired.
	
	The reverse implication that $w' \pfold{\tilde\phi_\sigma} x$ implies $w \pfold{\tilde\phi_\sigma} x$, follows by exchanging the letters $s$ and $t$.
\end{proof}

%
%
\section{Shadows}\label{sec:shadows}

We are finally able to introduce the notion of a shadow.

\begin{definition}[Shadows of words]\label{def:shadow}
	Let $(W,S)$ be a Coxeter system and $\phi$ be any orientation on $\Sigma(W,S)$. Then the \emph{shadow} of a word $w$ in $S$ with respect to $\phi$ is defined as follows
	\[
	\Shadow_{\phi}(w)=\{u\in W\;\vert\; w\pfold{\phi} u \}.
	\]
	In case $\phi$ is braid invariant, we may define $\Shadow_{\phi}(x)\define \Shadow_{\phi}(w)$ for any choice of a minimal expression $w$ for $x\in \aW$. We will sometimes write $\Shadow_\phi(c)$ for $\Shadow_{\phi}(x)$ when $c$ is the alcove corresponding to $x$.   
\end{definition}

\begin{figure}[h]
	\begin{minipage}[l]{0.45\textwidth}
		\begin{overpic}[width=0.9\textwidth]{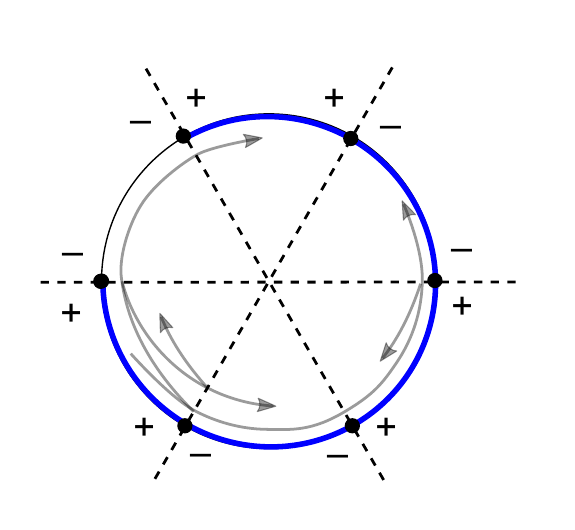}
			\put(15,25){\makebox(0,0)[cb]{$\id$}}%
			\put(80,60){\makebox(0,0)[cb]{$w_0$}}%
		\end{overpic}
	\end{minipage}
	\begin{minipage}[r]{0.45\textwidth}
		\begin{overpic}[width=0.9\textwidth]{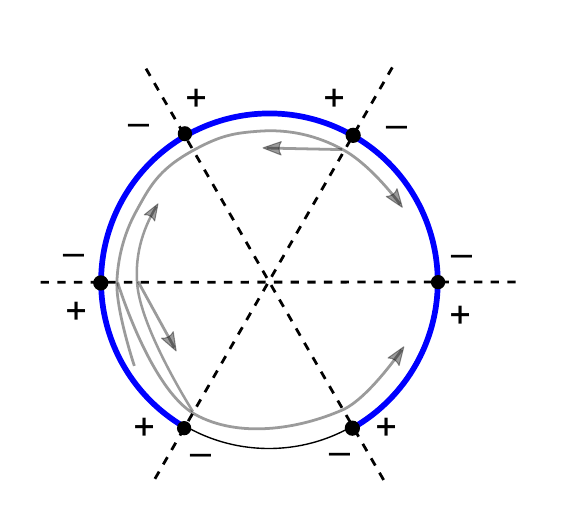}
			\put(15,25){\makebox(0,0)[cb]{$\id$}}%
			\put(80,60){\makebox(0,0)[cb]{$w_0$}}%
		\end{overpic}
	\end{minipage}
	\caption{The picture shows a non-braid-invariant orientation which hence produces different shadows (shown fat blue) for the two minimal galleries from $\id$ to $w_0$. See Example~\ref{ex:non-bi} for details. }
	
	\label{fig:not braid invariant}
\end{figure}

\begin{example}[Examples of shadows]\label{ex:non-bi}
	In general the shadow will depend on the choice of a word representing $x$, as illustrated in Figure~\ref{fig:not braid invariant}. 
	The orientation on the type $A_2$ Coxeter complex shown here is such that the two minimal galleries from $\id$ to $w_0$shown in light gray (going clockwise vs counterclockwise) produce different shadows which are colored in on the respective complex in fat blue.  Hence this orientation is not braid invariant. In the Figure we draw both their positively folded images (also as gray paths)  and their shadows (as fat blue edges in the complex). 
\end{example}

See also Figure~\ref{fig:BruhatOrder} for some examples of shadows with respect to the trivial positive orientation.

\begin{definition}[Regular and full shadows]
	Let $\aW$ be an affine Weyl group. Define for any $x \in \aW$ and any Weyl chamber orientation $\phi_a$ with $a\in\sW$  the \emph{regular shadow} of $x$ with respect to $a$ to be 
	\[
	\Shadow_a(x) := \Shadow_{\phi_a}(w)= \{y \in W : x \pfold{\phi_a} y\}
	\]
	for any minimal word $w$ with $[w]=x$.
	We define the \emph{full shadow} of $x$ to be the following union of regular shadows
	\[
	\Shadow(x) := \bigcup_{a\in\sW} \Shadow_a(x).
	\]
\end{definition}

The importance of full shadows will become clear in applications presented in \cite{MNST}  and \cite{shadows-applications}. 

\begin{figure}[h]
	\scalebox{-1}[1]{\includegraphics[width=0.5\textwidth, angle=180]{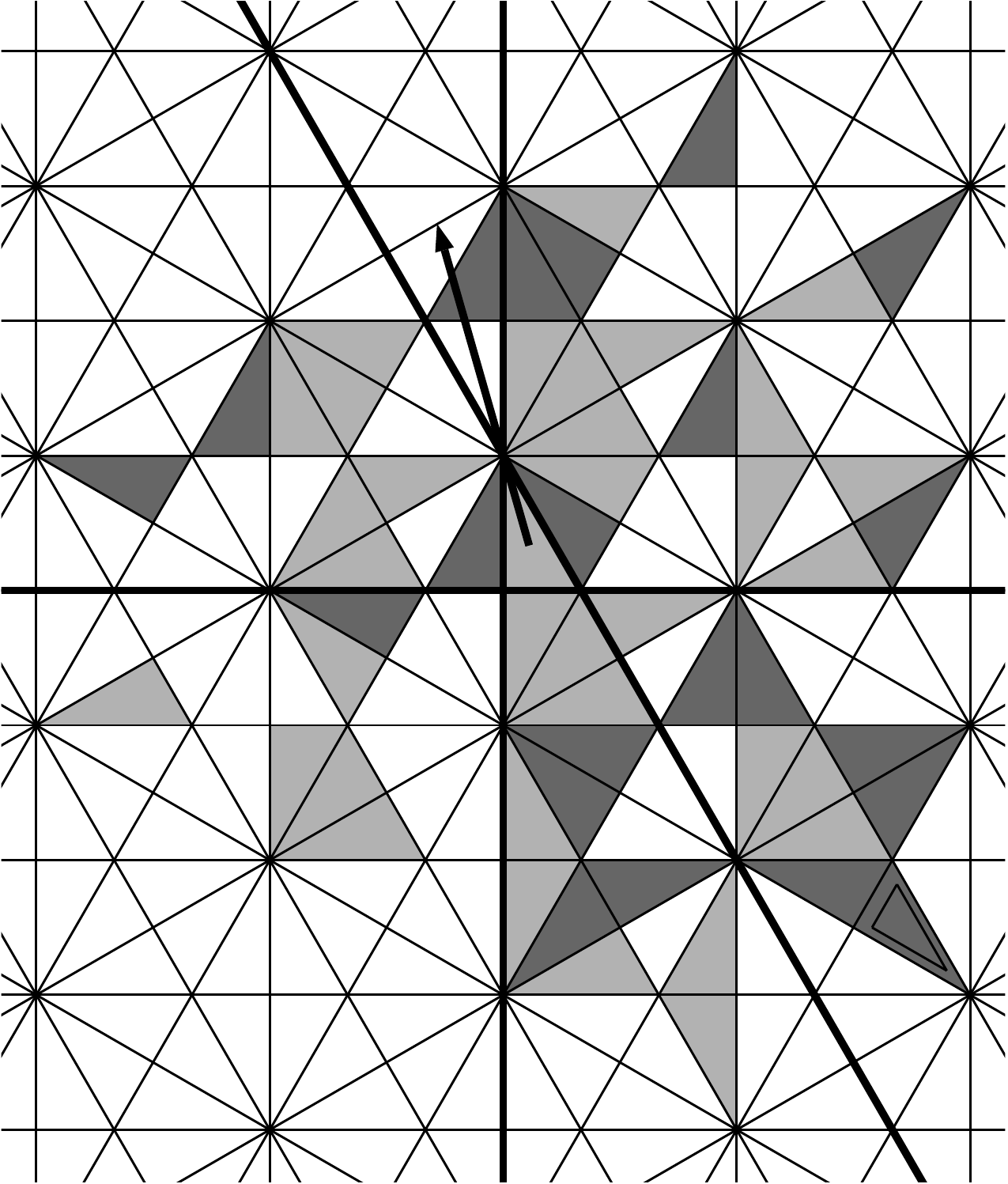}}
	
	\caption{The regular vector in the figure determines a Weyl chamber orientation. This picture shows the full and regular shadows with respect to that orientation in a type $\tilde G_2$ Coxeter group.  For details refer to Example~\ref{ex:shadows}. }
	\label{fig:shadows-soft-hard-2}
\end{figure}

\begin{remark}[Regularity]
	Regular shadows are determined by a choice of an equivalence class of  Weyl chambers in $\Sigma$ or, equivalently, a choice of a Weyl chamber at infinity. This corresponds to a regular direction (i.e. regular vector based at $0$) contained in the unique Weyl chamber representing the class that is based at $0$.  	Hence the term \emph{regular} shadow. 
\end{remark}

\begin{example}[Regular versus full shadows]\label{ex:shadows}
	In Figure~\ref{fig:shadows-soft-hard-1} and~\ref{fig:shadows-soft-hard-2}  we illustrate full and regular shadows of elements in type $\tilde A_2$ and $\tilde G_2$. In both figures the set of all shaded alcoves is the full shadow $\Shadow(c)$ of the outlined alcove $c$. The dark shaded alcoves are the elements of the regular shadow of the outlined element with respect to the orientation defined by the chamber at infinity to which the arrow points. 
\end{example}

\begin{remark}[Shadows vs retractions]
There is another vary natural geometric interpretation of shadows. Namely, one can show that a shadow of an element $x$ in some affine Coxeter group $\aW$ with respect to some orientation based at infinity can be interpreted in terms of a thick affine building of the same type as $\aW$. The shadows is the same as the image of a retraction from the same direction at infinity of the pre-image of a second type of retraction. This connection was already hidden in \cite{KostantConvexity} and will be made explicit in \cite{MNST}.  
\end{remark}

In the next proposition we formally summarize that indeed intervals of the form $[\id, x]$ in Bruhat order can be described via shadows. This is easily seen using the description of Bruhat order via the subword property.  

\begin{remark}[Subword property]
	The subword property (see \cite[Thm 2.2.2]{BjoernerBrenti}) implies that one can describe the Bruhat order as follows. Let $w=s_1s_2\ldots s_n$ be a reduced expression for $x=[w]$ and let $y\in W$. Then
	\[
	y\leq x \Leftrightarrow  
	\left. \begin{array}{cc} \text{ there exists a reduced expression $u$ for $y$ with } \\ 
	u=s_{i_1}s_{i_2}\ldots s_{i_k}, 1\leq i_1 < i_2< \ldots  < i_k \leq n.  \end{array} \right. 
	\]	
	That is $y\leq x$ if and only if for any reduced expression $w$ for $x$ there exists a reduced expression $u$ for $y$ which appears as a subword of $w$. 
\end{remark}

\begin{prop}[Bruhat order and shadows]\label{prop:bruhat_by_folding}
	Let $\phi_+$ be the trivial positive orientation and let $\phi_\id$ be the alcove orientation towards $\id$.  For any pair of elements $x,y \in W$ one has 
	\[
	x\geq y \; \Leftrightarrow \; x \pfold{\phi_+} y \; \Leftrightarrow\;  x \pfold{\phi_\id} y . 
	\]
	In particular $\Shadow_{\phi_+}(x)=\Shadow_{\phi_\id}(x)=[\id,x].$ 
\end{prop}
\begin{proof}
	From the sub-word property the first equivalence is obvious as reduced expressions are in bijection with minimal galleries. It is also obvious that $(x \pfold{\phi_\id} y \; \Rightarrow \; x \pfold{\phi_+} y)$, since any $\phi_\id$-positive folding of some gallery is also $\phi_+$-positive.
	
	To show $(x \pfold{\phi_+} y \; \Rightarrow \; x \pfold{\phi_\id} y)$: Let $w$ be a reduced expression for $x$, let $n = \ell(x)$. Among all $I \subset \{1, \ldots, n\}$ such that $\gamma_w^I$ ends at $c_y$, choose $I$ such that the sum of its elements is minimal. This ensures that $\gamma_w^I$ is $\phi_\id$-positively folded, because if $\gamma_w^I$ were not positively folded at $i \in I$, we could replace $i$ with some smaller value $j$. Specifically let $j$ be the first index such that the $i$-th and $j$-th panels of $\gamma_w^I$ lie in the same hyperplane. Then for $J := I \Delta {i,j}$ we have $\gamma_w^J$ ending at $c_y$ and $\sum(J) < \sum(I)$ because $j < i$, since every gallery from $\id$ crosses all hyperplanes from the $\phi_\id$--positive to the $\phi_\id$--negative side first. Compare also with \cite[Lemma 2.2.1]{BjoernerBrenti}.
\end{proof}

\begin{example}[Bruhat order and shadows]\label{ex:BruhatOrder}
	The shaded alcoves in Figure~\ref{fig:BruhatOrder} are the elements of the shadow of $x$ with respect to the trivial positive orientation on a type $\tilde A_2$ Coxeter complex. By the previous proposition this is the same as the Bruhat interval $[\id, x]$ and also the same as $\Shadow_\id(x)$.  
\end{example}

\begin{figure}[h]
	\begin{minipage}[l]{0.45\textwidth}
		\begin{overpic}[width=0.9\textwidth]{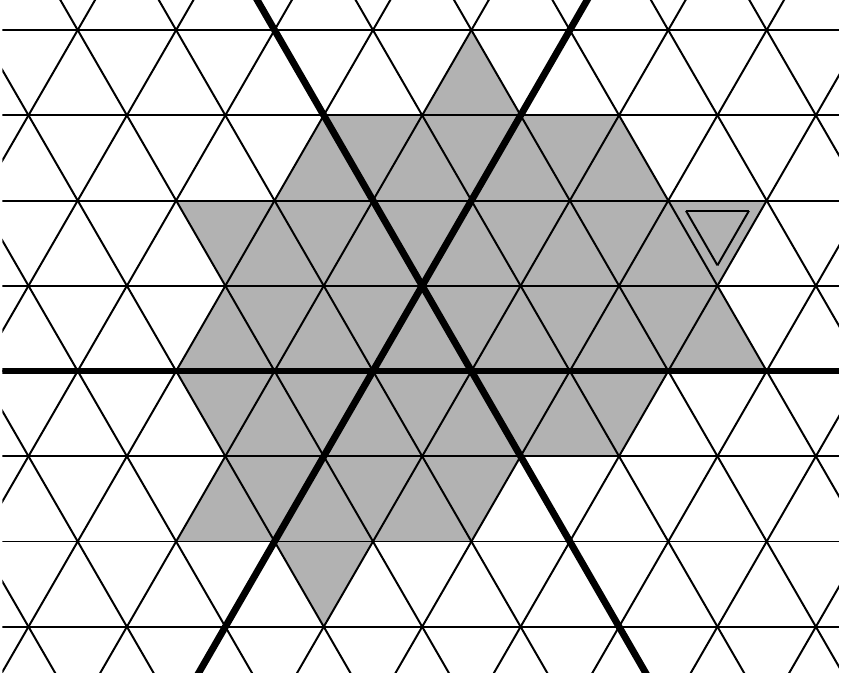}
			\put(50,37){\makebox(0,0)[cb]{$\id$}}%
			\put(85,50){\makebox(0,0)[cb]{$x$}}%
		\end{overpic}
	\end{minipage}
	\begin{minipage}[r]{0.45\textwidth}
		\begin{overpic}[width=0.9\textwidth]{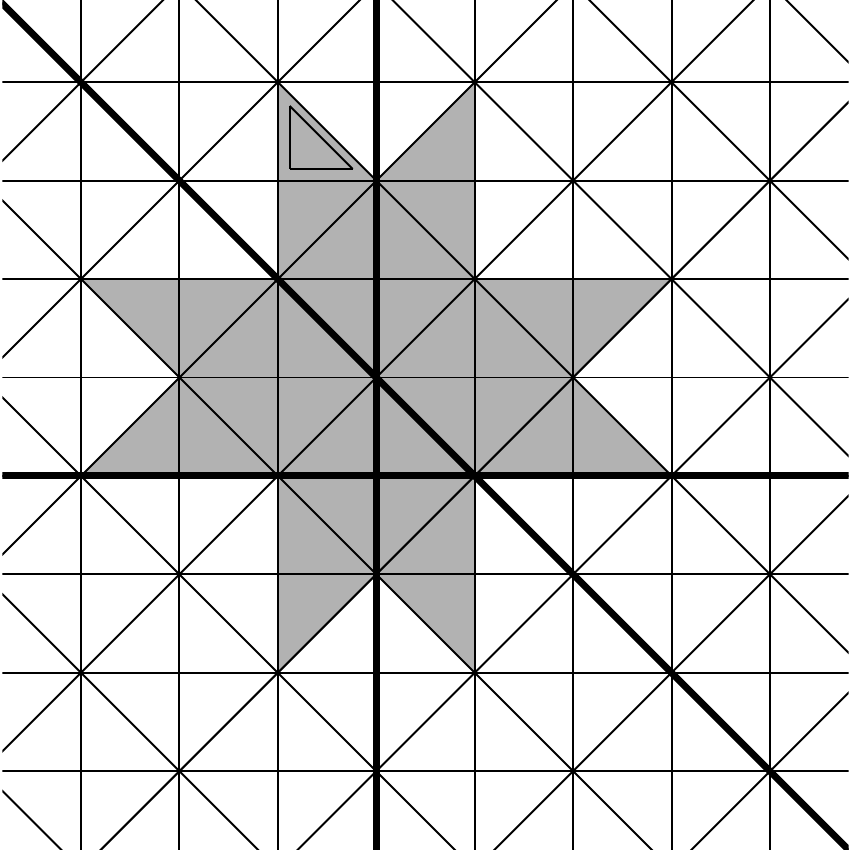}
			\put(48,45){\makebox(0,0)[cb]{$\id$}}%
			\put(36,81){\makebox(0,0)[cb]{$x$}}%
		\end{overpic}
	\end{minipage}
	\caption{The picture shows shadows $\Shadow_{\phi_+}(x)$ with respect to the trivial positive orientation $\phi_+$ in types $\tilde A_2$ (left) and $\tilde B_2$ (right). See also Example~\ref{ex:BruhatOrder} for different interpretations. }
	\label{fig:BruhatOrder}
	
\end{figure}

\begin{remark}[Other intervals in Bruhat order]
	Note that it is also possible to express intervals $[a,b]$, for $a,b\in W$, in Bruhat order in terms of positively folded galleries. To do this one needs to consider elements/alcoves $c$ in $\Shadow_+(b)$ that fold onto $a$. 
\end{remark}

%
%
\section{Recursive computation of regular shadows}\label{sec:regular shadows}
In this section, we examine the properties of regular shadows (and full) shadows and prove two identities in Theorems~ \ref{thm:regular_shadow} and \ref{thm:partial_shadow} from which we obtain two algorithms that are well suited to compute regular and full shadows. Suppose throughout the section, that $(\aW,S)$ is affine. 

\subsection{Structural results}

In the following we mean by a \emph{direction in \aW}, denoted by $\varphi\in\Dir(\aW)$, a chamber in the boundary $\partial\Sigma$. That is $\Dir(\aW)\define\Ch(\partial\Sigma(\aW, S))$.
By what we have discussed at the end of Section~\ref{sec:Coxeter} directions are in natural bijection with elements in $\sW$. 
Each direction induces a Weyl chamber orientation $\widetilde \phi_\varphi$ on $\Sigma$. We will abbreviate $\Shadow_{\widetilde\phi_\varphi}(x)$ by $\Shadow_\varphi(x)$. 

Note that the condition $\v_\varphi(s) < 0$ (resp. $>0$) in the next theorem simply means that the alcove corresponding to $s$ is on the negative (resp. positive)  side of the hyperplane separating $s$ from $\id$. 

\begin{thm}[Recursive computation of regular shadows]\label{thm:regular_shadow}
	For every $\varphi \in \Dir(\aW)$, all $x \in \aW$ and $s \in S$ the following holds.
	\begin{enumerate}[label=(\roman*)]
		\item\label{item:DR} If $s$ is in the right descent set $D_R(x)$ of $x$, then 
		\[
		\Shadow_\varphi(x) = \Shadow_\varphi(xs) \cdot s \cup \{z \in \Shadow_\varphi(xs): \v_\varphi(zs) < \v_\varphi(z)\}.
		\]
		\item\label{item:DL} If $s$ is in the left descent set $D_L(x)$ of $x$, then 
		\[
		\Shadow_\varphi(x) = \left\{\begin{array}{ll} 
		s \cdot \Shadow_{s\varphi}(sx) \cup \Shadow_\varphi(sx) & \text{ if } \v_\varphi(s) < 0  \\	
		s \cdot \Shadow_{s\varphi}(sx) & \text{ if } \v_\varphi(s) > 0.  
		\end{array} \right.
		\]
	\end{enumerate}
\end{thm}
\begin{proof}
	In this proof we will not distinguish between alcoves and the group elements labeling them. 
	
	To prove item \ref{item:DR} suppose that $s \in D_R(x)$. Let $w'$ be a reduced expression  for $xs$. Since $\ell(xs) < \ell(x)$ the word $w := w's$ is a reduced expression for $x$.
	
	We first prove "$\subseteq$": Let $y \in \Shadow_\varphi(x)$. Then there exists a $\varphi$-positively folded gallery $\gamma = (c_0 = \id, p_1, \ldots, c_{n-1}, p_n, c_n = y)$ of type $w$ from $\id$ to $y$.
	
	\emph{Case 1:} If $c_n = c_{n-1}$, then removing the last alcove of $\gamma$ yields a gallery of type $w'$ from $\id$ to $y$, so $y \in \Shadow_\varphi(xs)$. By $\varphi$-positivity of $\gamma$, $c_n = y$ lies on the $\varphi$-positive side of $p_n$. Since $p_n$ is of type $s$, the panel $p_n$ lies on the hyperplane $H_r$ corresponding to the reflection $r = ysy^{-1}$. By Lemma~\ref{lem:reflection-v}, this implies that $\v_\varphi(y) > \v_\varphi(ry)$, and since $ry = ysy^{-1}y = ys$, we obtain  that $y \in \{z \in \Shadow_\varphi(xs): \v_\varphi(zs) < \v_\varphi(z)\}$.
	
	\emph{Case 2}: If $c_n = c_{n-1}s$, then removing the last alcove of $\gamma$ yields a gallery of type $w'$ from $\id$ to $ys$, so $ys \in \Shadow_\varphi(xs)$ and thus $y \in \Shadow_\varphi(xs) \cdot s$.
	
	To see the converse containment "$\supseteq$" let $y \in \Shadow_\varphi(xs) \cdot s$. Then $xs \pfold{\varphi} ys$, so there exists a $\varphi$-positively folded gallery $\gamma = (c_0 = \id, p_1, \ldots, p_n, c_n = ys)$ of type $w'$. Now since the alcove $c_n = ys$ and $y$ meet in a panel $p$ of type $s$, we may extend $\gamma$ to the gallery $(c_0, p_1, \ldots, p_n, c_n, p, y)$ which is $\varphi$-positively folded from $\id$ to $y$ of type $w's = w$, so $y \in \Shadow_\varphi(x)$.
	
	Now let $y \in \{z \in \Shadow_\varphi(xs): \v_\varphi(zs) < \v_\varphi(z)\}$. Then $xs \pfold{} y$, so there exists a $\varphi$-positively folded gallery $\gamma = (c_0 = \id, p_1, \ldots, p_n, c_n = y)$ of type $w'$. Now let $p$ be the panel of $y$ of type $s$. Then $p$ lies in the hyperplane $H_r$ corresponding to the reflection $r := ysy^{-1}$. Since $ry = ysy^{-1} = ys$, we have that $\v_\varphi(ry) \leq \v_\varphi(y)$, thus $y$ lies on the positive side of $H_r$ and the gallery $(c_0, p_1, \ldots, p_n, c_n, p, y)$ is a $\varphi$-positively folded gallery of type $w's = w$ (thus of type $x$) from $\id$ to $y$.
	
	We split the proof of item \ref{item:DL} into two cases and assume first that $s \in D_L(x)$ with $\v_\varphi(s) > 0$.
	Let $w'$ be a reduced word for $sx$. Put  $w = sw'$. Since $\ell(sx) \leq \ell(x)$, the word $w$ is a reduced expression for $x$.
	
	Consider "$\subseteq$": Let $y \in \Shadow_\varphi(x)$. Then there is a $\varphi$-positively folded gallery of type $w$ from $\id$ to $y$.
	
	\emph{Case a:} Suppose $c_1 = s$. Define a sub-gallery $\gamma' = (c_1, p_2,  \ldots, p_n, c_n)$ of $\gamma$. 
	Then $\gamma'$ is $\varphi$-positively folded from $s$ to $y$, so by Lemma~\ref{lem:left-action}  the gallery $s\gamma'$ is $s\varphi$-positively folded of type $w'$ from $\id$ to $sy$. Therefore $sy \in \Shadow_{s\varphi}(sx)$ and $y \in s \cdot \Shadow_{s\varphi}(sx)$.
	
	\emph{Case b:} Suppose $c_1 = \id$. Then the sub-gallery $\gamma' = (c_1, p_2,  \ldots, p_n, c_n)$ of $\gamma$ is $\varphi$-positively folded of type $w'$ from $\id$ to $y$, so $y \in \Shadow_\varphi(sx)$.
	
	To see "$\supseteq$" let $y \in s \cdot \Shadow_{s\varphi}(sx)$. Then there exists a $s\varphi$-positively folded gallery $\gamma$ of type $w'$ from $\id$ to $sy$. By Lemma~\ref{lem:left-action}, the gallery $s\gamma$ is $\varphi$-positively folded of type $w'$ from $s$ to $y$. Let $p$ be the panel shared by alcoves $\id$ and $s$. The gallery $(\id, p, s)$ is now nonstammering of type $s$, therefore trivially $\varphi$-positively folded. This implies that extending the gallery $s\gamma$ at the front by $(\id, p, s)$ yields a gallery  $(\id, p, s) + s\gamma$ which is also $\varphi$-positively folded and runs from $\id$ to $y$. Moreover, its type is $sw' = w$, proving that $y \in \Shadow_\varphi(x)$.
	
	Now let $y \in \Shadow_\varphi(sx)$. Let $\gamma$ be a $\varphi$-positively folded gallery of type $w'$ from $\id$ to $y$. 
	Let $p$ be the panel shared by alcoves $\id$ and $s$. Since $\v_\varphi(s) < 0 = \v_\varphi(\id)$, we know that $\id$ lies on the $\varphi$-positive side of $p$ and thus the gallery $(\id, p, \id)$ is $\varphi$-positively folded of type $s$. Thus we may  extend $\gamma$ to a gallery  $(\id, p, \id) + \gamma$ which turns out to be the desired $\varphi$-positively folded gallery of type $sw' = w$ from $\id$ to $y$. Therefore $y \in \Shadow_\varphi(x)$.
	
	Assume for the second case of \ref{item:DL} that $s \in D_L(x)$ with $\v_\varphi(s) > 0$. Let $w'$ be a reduced expression for $sx$. Since $\ell(sx) \leq \ell(x)$, the word $w := sw'$ is a reduced expression for $x$.
	
	Let $y \in \Shadow_\varphi(x)$. There is a $\varphi$-positively folded gallery $\gamma = (c_0, p_1, \ldots, p_n, c_n)$ of type $w$ from $\id$ to $y$. Now $p_1$ is of type $s$ and lies on the hyperplane $H_s$, so if the alcove $s$ lies on the positive side of $H_s$ then $\id$ must lie on the negative side of $H_s$. Since $\gamma$ is positively folded, the alcove $c_1$ can not be equal to $\id$  and  therefore equals $s$. The gallery $\gamma' := (c_1, p_2, \ldots, p_n, c_n)$ is therefore a $\varphi$-positively gallery from $s$ to $y$ of type $w'$. So its image $s\gamma'$ is $s\varphi$-positively folded from $\id$ to $sy$ of type $w'$. This implies that $sy \in \Shadow_{s\varphi}(sx)$, so $y \in s \cdot \Shadow_{s\varphi}(sx)$.
	We have shown  "$\subseteq$". 
	
	We prove the opposite direction "$\supseteq$" as in the first case: let $y \in s \cdot \Shadow_{s\varphi}(sx)$. Then there exists a $s\varphi$-positively folded gallery $\gamma$ of type $w'$ from $\id$ to $sy$. By Lemma~\ref{lem:left-action}, the gallery $s\gamma$ is $\varphi$-positively folded of type $w'$ from $s$ to $y$. Let $p$ be the panel shared by alcoves $\id$ and $s$. The gallery $(\id, p, s)$ is now non stammering of type $s$ and therefore trivially $\varphi$-positively folded. So the extended gallery $(\id, p, s) + s\gamma$ is also $\varphi$-positively folded from $\id$ to $y$ of type $sw' = w$, proving that $y \in \Shadow_\varphi(x)$.
\end{proof}

We conclude this subsection  with a slightly more powerful variant of \ref{thm:regular_shadow} which we obtain by splitting up our regular shadows by translation class. 

\begin{definition}[Partial shadows]
	For an element $y\in \aW$ write $\bar{y}$ for its image in the spherical Weyl group $\sW$ under the natural projection. Then, for any $x \in \aW$, $a \in \sW$ and $\varphi \in \Dir(\aW)$ define the \emph{partial shadow in local direction $a$} to be the set 
	\[
	\Shadow_\varphi^a(x) := \{y \in \Shadow_\varphi(x) \;\vert\;  \bar{y} = a\}.
	\]
\end{definition}

\begin{thm}[Recursive computation of partial shadows]\label{thm:partial_shadow}
	Let $x, y \in \aW$ with $\ell(xy) = \ell(x) + \ell(y)$. Let $a \in \sW$ and $\varphi \in \Dir(\aW)$. Then
	
	\[ \Shadow_\varphi^a(xy) = \bigcup_{b \in \sW} \Shadow^b_{\varphi}(x) \cdot \Shadow_{b^{-1}\varphi}^{b^{-1}a}(y).\]
\end{thm}

\begin{proof}
	Let $w_1 = (s_1, \ldots, s_k)$ be a reduced expression for $x$ and suppose that $w_2 = (s_{k+1}, \ldots, s_n)$ is a reduced expression for $y$. Then $w = (s_1, \ldots, s_n)$ is a reduced expression for $xy$.
	
	To show forward inclusion, suppose $z \in \Shadow_\varphi^a(xy)$ and fix a gallery $(c_0 = \id, p_1, \ldots, p_n, c_n = z)$ of type $w$. Then $(c_0, p_1, \ldots, p_k, c_k)$ is a $\varphi$-positively folded gallery of type $w_1$ from $\id$ to $x' := c_k$, and $x'^{-1}(c_k, \ldots, p_n, c_n)$ is a $x'^{-1}\varphi$-positively folded gallery of type $w_2$ from $\id$ to $y' := x'^{-1}z$. Choosing $b$ equal to $\bar{x}'$, we find that $x' \in \Shadow_\varphi^b(x)$ and because $\bar{y}' = \bar{x}'^{-1}\bar{z} = b^{-1}a$, we find $y' \in \Shadow_{b^{-1}\varphi}^{b^{-1}a}(y)$, thus $z = x'y' \in \Shadow_\varphi^b(x) \cdot \Shadow_{b^{-1}\varphi}^{b^{-1}a}(y)$.
	
	To show reverse inclusion, suppose $z \in \Shadow_\varphi^b(x) \cdot \Shadow_{b^{-1}\varphi}^{b^{-1}a}(y)$ for some $b \in \sW$. Then $z = x'y'$ for some $x' \in \Shadow_\varphi^b(x), y' \in \Shadow_{b^{-1}\varphi}^{b^{-1}a}(y)$. Now there exists a $\varphi$-positively folded gallery $\gamma_1$ of type $w_1$ from $\id$ to $x'$ and a $b^{-1}\varphi$-positively folded gallery $\gamma_2$ of type $w_2$ from $\id$ to $y'$. Since $\bar{x}' = b$, we know that $x'\gamma_2$ is $\varphi$-positively folded from $x'$ to $x'y' = z$, so $\gamma = \gamma_1 + x'\gamma_2$ is $\varphi$-positively folded from $\id$ to $z$. Finally $\bar{z} = \bar{x}'\bar{y}' = bb^{-1}a = a$, therefore $z \in \Shadow_\varphi^a(xy)$.
\end{proof}

\subsection{Algorithms}

Much like intervals in Bruhat order have a recursive descriptions  Theorem~\ref{thm:regular_shadow} allows us to construct regular shadows recursively from regular shadows of left or right subwords.
We will now provide two algorithms. The first one uses the left-multiplication action of $\aW$ on itself and \ref{item:DR} of \ref{thm:regular_shadow}, the other uses the right-multiplication action and \cref{item:DL}.

\begin{lemma}[Algorithm L]\label{lem:algorithm_l}
	Fix a direction $\varphi\in\Dir(\aW)$ and let $x \in\aW$. Fix a reduced word  $w = (s_1, \ldots, s_n) \in S^*$ for $x$. 
	Put $A_0 = \{\id\}$ and define for $i = 1, \ldots, n$ the set  
	\[
	A_i := A_{i-1} \cdot s_i \cup \{z \in A_{i-1}  : v_\varphi(zs) < v_\varphi(z)\}.
	\]
	Then $A_n = \Shadow_\varphi(x)$.  
\end{lemma}

\begin{proof}
	It is easy to iteratively show by \ref{thm:regular_shadow} \ref{item:DR} that $A_i = \Shadow_\varphi(s_1 \cdots s_i)$ for $i = 0, \ldots, n$.
\end{proof}

\begin{remark}[]
	Note that since $z$ and $zs$ are only separated by the hyperplane $H_{zsz^{-1}}$, $\v_\varphi(z)$ and $\v_\varphi(zs)$ only differ by $p_\varphi(z, H_{zsz^{-1}}) - p_\varphi(zs, H_{zsz^{-1}})$, so $\v_\varphi(zs) < \v_\varphi(z)$ is equivalent to the fact that $z$ lies on the $\varphi$-positive side of the panel of $z$ of type $s$.
\end{remark}

Alternatively, we can use Lemma~\ref{lem:left-action} to see that $\v_\varphi(zs) < \v_\varphi(z)$ if and only if $\v_{z^{-1}\varphi}(s) < \v_{z^{-1}\varphi}(\id) = 0$. The latter is equivalent to $p_{z^{-1}\varphi}(s, H_s) = 0$.

\begin{lemma}[Algorithm R]\label{lem:algorithm_r}
	Let $x \in \aW$, and let $w = (s_n, \ldots, s_1) \in S^*$ be a reduced expression for $x$ (note the unusual indexing).
	
	For all $\varphi \in \Dir(\aW)$ let $B^\psi_0 := \{\id\}$. For $i = 1, \ldots, n$, and all $\varphi \in \Dir(\aW)$ let
	
	$$B^\varphi_i = \begin{cases}
	s_iB^{s_i\varphi}_{i-1} \cup B^\varphi_{i-1} & \text{ if } \v_\varphi(s_i) < 0, \\
	s_iB^{s_i\varphi}_{i-1} & \text{ if } \v_\varphi(s_i) > 0.
	\end{cases}$$
	
	Then $B^\varphi_n = \Shadow_\varphi(x)$ for all $\varphi \in \Dir(\aW)$.
\end{lemma}

\begin{proof}
	It is easy to iteratively show by \ref{thm:regular_shadow} \ref{item:DL} that $B^\varphi_i = \Shadow_\varphi(s_i \cdots s_1)$ for all $\varphi \in \Dir(\aW)$ and $i = 0, \ldots, n$.
\end{proof}

\subsection{Remarks on the computational effort} 

For a fixed orientation $\phi$ a simple yet inefficient algorithm to calculate the $\phi$-shadow of some element $x$ would be to take a minimal gallery $\gamma$ from $\id$ to $x$ and construct all $2^{\ell(x)}$ foldings $\gamma^I$ of $\gamma$. Then $\Shadow_\phi(x)$ is the set of endings of all the galleries in this set that are $\phi$-positively folded. 

Unfortunately this naive approach requires examining a number of foldings exponential in $\ell(x)$. One can immediately improve this to a polynomial-time algorithm by checking only the foldings of $\gamma$ with less than $k\define \ell(w_0)$ folds by Proposition~\ref{prop:bounds on folds} (recall that $w_0$ denotes the longest element in $\sW$). However, there are then still over $\binom{\ell(x)}{k} \mathop{\hat{\approx}} \ell(x)^k$ such foldings. So in case $\ell(w_0)$ is large this quickly becomes infeasible again. The algorithms L and R we constructed by means of Theorem~\ref{thm:regular_shadow} are more efficient.

Algorithm L can compute $A_i$ from $A_{i-1}$ using $\Theta(|A_i|)$ multiplications and $\Theta(|A_i|)$ evaluations of $p_\varphi(\cdot, \cdot)$. Since $A_is_{i+1} \cdots s_n \subset A_n = \Shadow_\varphi(x)$, the total calculation effort of Algorithm R is bounded by $\mathcal{O}(\ell(x)|\Shadow_\varphi(x)|)$. The shadow $\Shadow_\varphi(x)$  is a subset of $\{y \in \aW : \ell(y) \leq \ell(x)\}$. Hence  one can conclude from the deletion condition of Coxeter groups that the total calculation effort is bounded by $\mathcal{O}(\ell(x)\ell(x)^d) = \mathcal{O}(\ell(x)^{d+1})$. This is a potentially very large improvement over the $\Omega(\ell(x)^{\ell(w_0)})$ effort we get from our improved naive algorithm. 

Algorithm R can compute the $B^\varphi_i$ from all sets $B^\varphi_{i-1}$ using $\Theta(\sum_{\varphi \in \Dir(\aW)}|B^\varphi_i|)$ operations. Since $s_n \cdots s_{i+1} B_i^{s_n \cdots s_{i+1}\varphi} \subset B_n^\varphi = \Shadow_\varphi(x)$, the total calculation effort of Algorithm R is bounded by $\mathcal{O}(\sum_{\varphi \in \Dir(\aW)}\ell(x)|\Shadow_\varphi(x)|)$, which is the same effort as calculating all regular shadows of $x$ separately using Algorithm L.

The main difference between algorithms L and R is that Algorithm L iteratively calculates shadows in a single direction, while Algorithm R calculates shadows in all directions at once. If we want to calculate a single regular shadow of some element $x \in \aW$, then Algorithm L is preferable, especially when $\Dir(\aW)$ is large. If we want to find the full shadow of $x$, then we need the shadows for all directions anyway, so Algorithm R is preferable to repeated use of Algorithm L because Algorithm R requires much less checking whether certain chambers lie on positive sides of their panels.

\renewcommand{\refname}{Bibliography}
\bibliography{Bibliography}
\bibliographystyle{plain}
\end{document}